\documentclass[11pt, reqno]{amsart}

\usepackage{graphicx}
\usepackage{epic}
\usepackage{amsmath}
\usepackage{amsfonts}
\usepackage{amssymb}
\usepackage{amsthm}
\usepackage{eucal}
\usepackage[ansinew]{inputenc}
\usepackage[lmargin=3.0cm, rmargin=3.0cm, top=3.8cm, bottom=3.8cm]{geometry}

\newtheorem{theorem}{Theorem}[section]
\newtheorem{lemma}[theorem]{Lemma}
\newtheorem{corollary}[theorem]{Corollary}
\newtheorem{proposition}[theorem]{Proposition}

\numberwithin{equation}{section}

\begin{document}

\title[Non-intersecting Brownian Motions from $q$ to $p$ points]{Non-intersecting Brownian Motions leaving from and going to several points}

\author[M.~Adler]{Mark Adler}
\address{M.A.\\ Department of Mathematics, Brandeis University \\415 South Street \\MS 050, Waltham, MA 02454}
\email{adler@brandeis.edu}

\author[P.~van Moerbeke]{Pierre van Moerbeke}
\address{P.v.M.\\ Department of Mathematics, Université Catholique de
Louvain\\ Chemin du Cyclotron 2\\ 1348 Louvain-la-Neuve\\ Belgium}
\address{and}
\address{P.v.M.\\ Department of Mathematics, Brandeis University\\ 415 South Street\\ MS 050, Waltham, MA 02454}
\email{pierre.vanmoerbeke@uclouvain.be \\ vanmoerbeke@brandeis.edu}

\author[D.~Vanderstichelen]{Didier Vanderstichelen}
\address{D.V.\\ Department of Mathematics, Université Catholique de
Louvain\\ Chemin du Cyclotron 2\\ 1348 Louvain-la-Neuve\\ Belgium}
\email{didier.vanderstichelen@uclouvain.be}

\date{April 22 2011}
\thanks{The first and second authors gratefully acknowledge the support of a National Science Foundation grant \#
DMS-07-04271. The second author gratefully acknowledges the support of a European
Science Foundation grant (MISGAM), a Marie Curie Grant
(ENIGMA), and a FNRS grant. The second and third authors gratefully aknowledge the partial support of the Belgian Interuniversity Attraction Pole P06/02. The third author is a Research Fellow at F.R.S.-FNRS, Belgium.}

\begin{abstract} Consider n non-intersecting Brownian motions on $\mathbb{R}$, depending on time $t \in [0,1]$, with $m_i$ particles forced to leave from $a_i$ at time $t=0$, $1\leq i\leq q$, and $n_j$ particles forced to end up at $b_j$ at time $t=1$, $1\leq j\leq p$. For arbitrary $p$ and $q$, it is not known if the distribution of the positions of the non-intersecting Brownian particles at a given time $0<t<1$, is the same as the joint distribution of the eigenvalues of a matrix ensemble. This paper proves the existence, for general $p$ and $q$, of a partial differential equation (PDE) satisfied by the log of the probability to find all the particles in a disjoint union of intervals $E=\cup_{i=1}^{r}[c_{2i-1},c_{2i}]\subset\mathbb{R}$ at a given time $0<t<1$. The variables are the coordinates of the starting and ending points of the particles, and the boundary points of the set $E$. The proof of the existence of such a PDE, using Virasoro constraints and the multicomponent KP hierarchy, is based on the method of elimination of the unwanted partials; that this is possible is a miracle! Unfortunately we were  unable to find its explicit expression. The case $p=q=2$ will be discussed in the last section.
\end{abstract}
\maketitle

\tableofcontents


\section{Introduction}

Consider $N$ independent Brownian motions on the real line, starting at time $t=0$ and ending at time $t=1$ at prescribed positions, and conditioned not to intersect during the time interval $]0,1[$. In the particular case when all the Brownian motions start at the same position at $t=0$ and end at the same position at $t=1$, let's say the origin, the positions of the Brownian particles at an intermediate time $0<t<1$ have the same distribution, after a space-time transformation, as the eigenvalues of a randomly chosen matrix of the GUE ensemble. It is also the distribution of $N$ Dyson Brownian motions on the real line. This process, discovered by Dyson \cite{F.D.}, describes the motion in time of the eigenvalues of a $N\times N$ Hermitian matrix whose real and imaginary parts of the entries perform independent Ornstein-Uhlenbeck-processes, with an initial distribution given by the invariant measure of the process. In the case of one starting position and two (or more) ending positions, the positions of the Brownian particles at an intermediate time $0<t<1$ have the same distribution, after a space-time transformation, as the eigenvalues of a randomly chosen matrix of the Gaussian Hermitian ensemble with external source. The relationship between non-intersecting Brownian motions and matrix models has been developped by Johansson \cite{K.J.}, and for the Gaussian Hermitian matrix ensemble with external source by Aptekarev-Bleher-Kuijlaars \cite{A.B.K.}. See also Adler-Delépine-van Moerbeke \cite{A.D.V.M.} and Katori-Tanemura \cite{K.T.1,K.T.2} for a detailed description of the relationship between Dyson Brownian motions, non-intersecting Brownian motions and Gaussian Hermitian matrix ensembles. In particular, in \cite{K.T.1} both stochastic processes are obtained as scaling limits of the vicious walkers model. In the two particular cases cited (i.e. non-intersecting Brownian motions with one starting position and one or several ending positions), the relationship between non-intersecting Brownian motions and Hermitian matrix models has led to a deeper comprehension of the diffusion problems. In both cases, partial differential equations (PDE) for the finite $N$ diffusions have been obtained (see \cite{A.S.V.M.,A.V.M.1,A.V.M.V.}). For large $N$, upon taking appropriate scaling limits, different processes appear describing the transition probabilities of critical infinite dimensional diffusions, like the Airy process, the Sine process and the Dyson process (see \cite{A.V.M.4,K.T.2,T.W.1}) for one starting and ending position, and for two or more ending positions the Pearcey process (see \cite{A.B.K.,T.W.2}), the Airy process with $k$ outliers (see \cite{A.D.V.M.}), etc. The approach of Adler-van Moerbeke is to take scaling limits of the PDE's describing the finite $N$ processes to obtain PDE's for the transition probabilities of the critical infinite dimensional diffusions.

Consider now $N$ non-intersecting Brownian motions on the real line starting at time $t=0$ at $q$ and ending at time $t=1$ at $p$ prescribed positions, with $p,q\geq 2$. It is not known if there exists a matrix ensemble whose joint eigenvalue distribution describes the distribution of the positions of the Brownian motions at an intermediate time $0<t<1$. For $p=q=2$ this problem has first been studied by Daems-Kuijlaars \cite{Da.K.} and Daems-Kuijlaars-Veys \cite{Da.K.V.}. In these papers, the authors consider $N/2$ particles going from $a$ to $b$, and $N/2$ particles going from $-a$ to $-b$. They show that the correlation functions of the positions of the non-intersecting Brownian motions have a determinantal form, with a kernel that can be expressed in terms of mixed multiple Hermite polynomials. They analyze the kernel in the large $N$ limit, for a small separation of the starting and ending positions (i.e. when the product $ab$ is sufficiently small), and find the limiting mean density of particles is supported by one or two intervals. Taking usual scaling limits of the kernel in the bulk and near the edges they find the Sine and the Airy kernel. For large separation of the starting and ending positions, those results have been extended by Delvaux-Kuijlaars \cite{D.K.1}. In \cite{A.F.V.M.}, Adler-Ferrari-van Moerbeke study a similar situation, but with an asymmetric number of paths in the left and right starting and ending positions. Recently, Adler-Ferrari-van Moerbeke \cite{A.F.V.M.2} and also Delvaux-Kuijlaars-Zhang \cite{D.K.Z.} (see also \cite{D.K.2}) analyzed the large $N$-limit in a critical regime where the paths fill two tangent ellipses in the time-space plane. Using an appropriate double scaling limit, they prove the existence of a new process describing the diffusion of the particles near the point of tangency.

It seems to be a highly non-trivial problem to obtain concrete results about the processes describing the critical infinite dimensional diffusions, obtained as limiting situations of the problem of $N$ non intersecting Brownian motions on the real line starting at $q$ and ending at $p$ prescribed positions, with $p,q\geq 2$. The aim of this paper is to provide a better understanding of the finite $N$ diffusion for two or more starting and ending positions. We consider $N$ non-intersecting Brownian motions $x_1(t),\dots,x_N(t)$ on $\mathbb{R}$, starting at time $t=0$ in $q$ different points, and arriving in $t=1$ in $p$ different points. If $\mathbb{P}_{b_1,\dots,b_p}^{a_1,\dots, a_q}\big(\text{all}\ x_i(t)\in E\big)$ denotes the probability to find all the particles in a set $E$ at an intermediate time $0<t<1$, we prove the following theorem. 
\begin{theorem}\label{Theorem introduction}
For each value of the parameters $p\geq 1$ and $q\geq 1$, let $K^*$ be the smallest positive integer such that
\begin{align}
(x^2-3x+4)(K^*)^2+(-x^2+3x+4)K^*-2x(x^2-2x-1)> 0,\notag
\end{align}
with $x=p+q$. Let $E$ be a finite union of intervals. Under the assumptions $a_1+\dots+a_q=0$ and $b_1+\dots+b_p=0$, the function $\log \mathbb{P}_{b_1,\dots,b_p}^{a_1,\dots, a_q}\big(\text{all}\ x_i(t)\in E\big)$ satisfies a nonlinear PDE of order $K^*+3$ or less, the variables being the coordinates of the endpoints of the set $E$, and the coordinates of $a_1,\dots,a_q$ and $b_1,\dots,b_p$.
\end{theorem}

For example, for $4\leq x\leq 8$, the value of $K^*$ in this theorem is given in the following table :
\begin{center}
\begin{tabular}[c]{|c|c|c|c|c|c|}
\hline
$\mathbf{x}$ & $4$ & $5$ & $6$ & $7$ & $8$\\
\hline
$\mathbf{K}^*$ & $3$ & $4$ & $5$ & $5$ & $5$\\
\hline
\end{tabular}
\end{center}

The proof of Theorem \ref{Theorem introduction} will be given in section 5, and is based on the use of a particular integrable hierarchy, and Virasoro constraints. The use of these methods is suggested by the fact that the probability $\mathbb{P}_{b_1,\dots,b_p}^{a_1,\dots, a_q}\big(\text{all}\ x_i(t)\in E\big)$ has different descriptions (see section 2):
\begin{enumerate}
	\item It can be written, after making a space and time transformation, as a block moment matrix
	\begin{align}
	\mathbb{P}_{b_1,\dots,b_p}^{a_1,\dots, a_q}\big(\text{all}\ x_i(t)\in E\big)=\frac{1}{Z_N}\det\Bigg[\Big(\left\langle x^m\psi_i(x)\Big|y^n\varphi_j(y) \right\rangle\Big)_{\substack{0\leq m\leq m_i-1 \\ 0\leq n\leq n_j-1}}\Bigg]_{\substack{1\leq i\leq q \\ 1\leq j\leq p}},\label{block moment}
	\end{align}
	where $\psi_i(x)=e^{\tilde{a}_ix}$, $\varphi_j(y)=e^{\tilde{b}_jy}$, and the following inner product
	\begin{align}
	\left\langle x^m\psi_i(x)\Big|y^n\varphi_j(y) \right\rangle=\int_{\tilde{E}}x^{m+n}e^{(\tilde{a}_i+\tilde{b}_j)x}e^{-\frac{x^2}{2}}dx.\notag
	\end{align}
	The $\sim$'s indicate that a space-time transformation has been performed (see section 2, formula (\ref{normalized problem})).
	
	\item It can be written as a sum of multiple integrals
	\begin{align}
&\mathbb{P}_{b_1,\dots,b_p}^{a_1,\dots, a_q}\big(\text{all}\ x_i(t)\in E\big)\notag\\
&=\frac{1}{Z_N}\sum_{\sigma\in S_N}(-1)^{\sigma}\int_{\tilde{E}^N}\Big(\prod_{i=1}^{N}e^{-\frac{x_i^2}{2}}\,dx_i\Big)\Big(\Delta_{m_1}(x_1,x_2,\dots,x_{m_1})\prod_{i=1}^{m_1}\psi_1(x_{i})\Big)\notag\\
&\phantom{N!\int_{E^N}=}\times\dots\,\times\Big(\Delta_{m_q}(x_{m_1+\dots+m_{q-1}+1},\dots,x_{m_1+\dots+m_q})\prod_{i=1}^{m_q}\psi_q(x_{m_1+\dots+m_{q-1}+i})\Big)\notag\\
&\phantom{N!\int_{E^N}=}\times\Big[\Big(\Delta_{n_1}(x_{\sigma(1)},\dots,x_{\sigma(n_1)})\prod_{i=1}^{n_1}\varphi_1(x_{\sigma(i)})\Big)\notag\\
&\phantom{N!\int_{E^N}=}\times\dots\times\Big(\Delta_{n_p}(x_{\sigma(n_1+\dots+n_{p-1}+1)},\dots,x_{\sigma(n_1+\dots+n_p)})\prod_{i=1}^{n_p}\varphi_p(x_{\sigma(n_1+\dots+n_{p-1}+i)})\Big)\Big],\label{sum of integrals}
\end{align}
	where $\Delta_n$ is the Vandermonde determinant, and $S_N$ is the group of permutations of $N$ elements. 
\end{enumerate}

As shown in Adler-van Moerbeke-Vanhaecke \cite{A.V.M.V.}, the determinants of block moment matrices deformed in an appropriate way satisfy integrable hierarchies. Concretely, the determinant (\ref{block moment}) is deformed by adding exponentials containing additional families of time variables, one family for each weight function $\varphi_i$ and $\psi_j$, or equivalently, (\ref{sum of integrals}) is deformed by adding exponentials containing additional families of time variables, one family for each Vandermonde determinant. The determinants of the deformed block-moment matrices (\ref{block moment}) are then tau functions for the multi-component KP hierarchy. The multi-component KP hierarchy is a very general hierarchy of integrable equations, describing the time-evolution of matrix-valued pseudo-differential operators, depending on several families of time variables. These operators can be expressed in terms of so-called tau-functions, which encode the whole hierarchy. As a consequence, the determinants of the deformed block moment matrices satisfy some nonlinear PDE's. This is developped in section 3.

When $p=1$ or $q=1$, it is easy to see that all the terms in (\ref{sum of integrals}) are equal to each other, and the sum is simply $N!$ times a $N$-tuple integral over $\tilde{E}$. This unique integral corresponds to the joint eigenvalue probability of the Gaussian Hermitian ensemble with external source. It is a well known fact that matrix integrals deformed in an appropriate way satisfy Virasoro constraints (see \cite{A.S.V.M.}). These constraints are linear PDE's satisfied by the deformed integrals, involving a time part and a boundary part. Although we do not know if (\ref{sum of integrals}) for general $p$ and $q$ corresponds to (the reduction to polar coordinates of) a matrix integral, we show that each term in (\ref{sum of integrals}) separately satisfies Virasoro constraints. As a surprise, it appears that all the terms satisfy \emph{the same} Virasoro constraints, and hence, by linearity, it follows that (\ref{sum of integrals}) satisfies Virasoro constraints. This is developped in section 4.

Following a method developped by Adler-Shiota-van Moerbeke (see for example \cite{A.S.V.M.} and \cite{A.V.M.1}), Virasoro constraints with time and boundary parts can be used to eliminate all the partial derivatives with respect to the added time variables in the non-linear PDE's from the integrable hierarchy, and hence to obtain a non-linear PDE with respect to the variables of the unperturbed problem. The complexity of the problem studied in this paper does not enable one to perform concretely this elimination process and to obtain an explicit formula for arbitrary values $p,q>2$. It is a priori not even obvious at all that it converges to a PDE after a finite number of steps! In Theorem \ref{Theorem introduction} we prove, however, using a simple combinatorial argument, that it indeed does, and this for general $p$ and $q$. We would like to emphasize that the existence of a PDE satisfied by $\mathbb{P}_{b_1,\dots,b_p}^{a_1,\dots, a_q}\big(\text{all}\ x_i(t)\in E\big)$ is not obvious at all. Our proof rests on two surprising facts, the first being that the perturbed problem satisfies Virasoro constraints, and the second that the elimination process converges after a finite number of steps.

In the last section, we consider the case of non-intersecting Brownian motions with two starting and two ending positions. In this particular case, we improve considerably the general method given in the proof of Theorem \ref{Theorem introduction} to obtain a PDE.


\section{Non-intersecting Brownian motions with $q$ starting points and $p$ ending points}

Consider $N$ non-intersecting Brownian motions $x_1(t)$, $x_2(t)$, \dots, $x_N(t)$ in $\mathbb{R}$, leaving from distinct points $\alpha_1<\alpha_2<\dots<\alpha_N$ and forced to end up at distinct points $\beta_1<\beta_2<\dots<\beta_N$. From the Karlin-McGregor formula ~\cite{Karlin-McGregor}, we know that the probability that all $x_i(t)$ belong to a set $E\subset\mathbb{R}$ at a given time $0<t<1$ can be expressed in terms of the Gaussian transition probability
\begin{align}
p(t,x,y)=\frac{1}{\sqrt{\pi t}}\,e^{-\frac{(x-y)^2}{t}},\notag
\end{align}
in the following way
\begin{align}
\mathbb{P}_{\mathbf{\beta}}^{\mathbf{\alpha}}\big(\text{all }x_i(t)\in E\big)&:=\mathbb{P}\Bigg(\text{all }x_i(t)\in E\Bigg|\begin{array}{ll}\big(x_1(0),\dots,x_N(0)\big)=\big(\alpha_1,\dots,\alpha_N\big) \\ \big(x_1(1),\dots,x_N(1)\big)=\big(\beta_1,\dots,\beta_N\big) \end{array}\Bigg)\notag\\
&=\frac{1}{Z_N}\,\int_{E^N}\det\big[p(t,\alpha_i,x_j)\big]_{1\leq i,j\leq N}\,\det\big[p(1-t,x_i,\beta_j)\big]_{1\leq i,j\leq N}\,\prod_{i=1}^{N}dx_i\notag\\
&=\frac{1}{\tilde{Z}_N}\,\int_{E^N}\det\big[e^{\frac{2\alpha_ix_j}{t}}\big]_{1\leq i,j\leq N}\,\det\big[e^{\frac{2\beta_ix_j}{1-t}}\big]_{1\leq i,j\leq N}\prod_{i=1}^{N}e^{\frac{-x_i^2}{t(1-t)}}\,dx_i,\notag
\end{align}
where $Z_N$ and $\tilde{Z}_N$ are normalizing factors. In particular, if
\begin{align}
(\alpha_1,\dots,\alpha_N)\ &=\ \big(\underbrace{a_1,a_1,\dots,a_1}_{m_1},\underbrace{a_2,a_2,\dots,a_2}_{m_2},\dots,\underbrace{a_q,a_q,\dots,a_q}_{m_q}\big),\qquad a_1<a_2<\dots<a_q,\notag\\
(\beta_1,\dots,\beta_N)\ &=\ \big(\underbrace{b_1,b_1,\dots,b_1}_{n_1},\underbrace{b_2,b_2,\dots,b_2}_{n_2},\dots,\underbrace{b_p,b_p,\dots,b_p}_{n_p}\big),\qquad b_1<b_2<\dots<b_p,\notag
\end{align}
with $\sum_{i=1}^{q}a_i=\sum_{i=1}^{p}b_i=0$ and $\sum_{i=1}^{q}m_i=\sum_{i=1}^{p}n_i=N$, then we have
\begin{align}
\mathbb{P}_{b_1,\dots,b_p}^{a_1,\dots,a_q}\big(&\text{all }x_i(t)\in E\big)\notag\\
&:=\mathbb{P}\left(\text{all }x_i(t)\in E\Bigg|\begin{array}{llll}\big(x_1(0),\dots,x_N(0)\big)=\big(\underbrace{a_1,\dots,a_1}_{m_1},\dots,\underbrace{a_q,\dots,a_q}_{m_q}\big) \\ \big(x_1(1),\dots,x_N(1)\big)=\big(\underbrace{b_1,\dots,b_1}_{n_1},\dots,\underbrace{b_p,\dots,b_p}_{n_p}\big) \end{array}\right)\notag\\
&=\lim_{\substack{\alpha_1,\dots,\alpha_{m_1}\to a_1 \\ \dots \\ \alpha_{m_1+\dots+m_{q-1}+1},\dots,\alpha_{m_1+\dots+m_q}\to a_q \\ \beta_1,\dots,\beta_{n_1}\to b_1 \\ \dots \\ \beta_{n_1+\dots+n_{p-1}+1},\dots,\beta_{n_1+\dots+n_p}\to b_p}}\mathbb{P}_{\mathbf{\beta}}^{\mathbf{\alpha}}\big(\text{all }x_i(t)\in E\big).\notag
\end{align}
Consequently, we obtain
\begin{multline}
\mathbb{P}_{b_1,\dots,b_p}^{a_1,\dots,a_q}\big(\text{all }x_i(t)\in E\big)=\frac{1}{\hat{Z}_N}\int_{E^N}\prod_{i=1}^{N}e^{\frac{-x_i^2}{t(1-t)}}\,dx_i\notag\\
\times\,\det\left(\begin{array}{c}
\Big(x_j^ie^{\frac{2a_1x_j}{t}}\Big)_{\substack{0\leq i\leq m_1-1 \\ 1\leq j\leq N}} \\
 \\
\vdots \\
 \\
\Big(x_j^ie^{\frac{2a_qx_j}{t}}\Big)_{\substack{0\leq i\leq m_q-1 \\ 1\leq j\leq N}} 
\end{array}\right)\ .\,\det\left(\begin{array}{c}
\Big(x_j^ie^{\frac{2b_1x_j}{1-t}}\Big)_{\substack{0\leq i\leq n_1-1 \\ 1\leq j\leq N}} \\
 \\
\vdots \\
 \\
\Big(x_j^ie^{\frac{2b_px_j}{1-t}}\Big)_{\substack{0\leq i\leq n_p-1 \\ 1\leq j\leq N}} 
\end{array}\right).\notag
\end{multline}
We have, using the change of variables $x_i=\sqrt{\frac{t(1-t)}{2}}\,y_i$, $1\leq i\leq N$,
\begin{align}
\mathbb{P}_{b_1,\dots,b_p}^{a_1,\dots,a_q}\big(\text{all }x_i(t)\in E\big)=P_{p,q}\Big(\sqrt{\frac{2}{t(1-t)}}\,E;\sqrt{\frac{2(1-t)}{t}}\,a,\sqrt{\frac{2t}{1-t}}\,b\Big),\label{proba-normalized problem}
\end{align}
with the normalized problem being
\begin{align}
P_{p,q}(E;a,b):=\frac{1}{Z_{p,q}}\int_{E^N}\Big(\prod_{i=1}^{N}e^{\frac{-x_i^2}{2}}\,dx_i\Big)\,\det\big[\tilde{\psi}_i(x_j)\big]_{1\leq i,j\leq N}\,\det\big[\tilde{\varphi}_i(x_j)\big]_{1\leq i,j\leq N},\label{normalized problem}
\end{align}
where $Z_{p,q}$ is a normalizing factor, and where we have introduced the following notation
\begin{align}
\big(\tilde{\psi}_1(x),\dots,\tilde{\psi}_N(x)\big):=&\Big(e^{a_1x},xe^{a_1x},\dots,x^{m_1-1}e^{a_1x},e^{a_2x},xe^{a_2x},\dots,\notag\\
&\phantom{\Big(e^{a_1x}}x^{m_2-1}e^{a_2x},\dots,e^{a_qx},xe^{a_qx},\dots,x^{m_q-1}e^{a_qx}\Big),\notag\\
\big(\tilde{\varphi}_1(x),\dots,\tilde{\varphi}_N(x)\big):=&\Big(e^{b_1x},xe^{b_1x},\dots,x^{n_1-1}e^{b_1x},e^{b_2x},xe^{b_2x},\dots,\notag\\
&\phantom{\Big(e^{b_px}}x^{n_2-1}e^{b_2x},\dots,e^{b_px},xe^{b_px},\dots,x^{n_p-1}e^{b_px}\Big).\notag
\end{align}
In the following proposition, we consider a general situation, of which (\ref{normalized problem}) is a special case by setting $V(x)=\frac{x^2}{2}$, $\psi_i(x)=e^{a_i x},\  1\leq i\leq q$, and $\varphi_i(x)=e^{b_i x},\  1\leq i\leq p$.
\begin{proposition}\label{proposition 1}
Given an arbitrary potential $V(x)$ and arbitrary functions $\psi_1(x),\dots, \psi_q(x)$ and $\varphi_1(x),\dots,\varphi_p(x)$, define ($N=m_1+\dots+m_q=n_1+\dots n_p$)
\begin{align}
\big(\tilde{\psi}_1(x),\dots,\tilde{\psi}_N(x)\big):=&\Big(\psi_1(x),x\psi_1(x),\dots,x^{m_1-1}\psi_1(x),\psi_2(x),x\psi_2(x),\notag\\
&\dots,x^{m_2-1}\psi_2(x),\dots,\psi_q(x),x\psi_q(x),\dots,x^{m_q-1}\psi_q(x)\Big),\notag\\
\big(\tilde{\varphi}_1(x),\dots,\tilde{\varphi}_N(x)\big):=&\Big(\varphi_1(x),x\varphi_1(x),\dots,x^{n_1-1}\varphi_1(x),\varphi_2(x),x\varphi_2(x),\dots,x^{n_2-1}\varphi_2(x),\dots,\notag\\
&\phantom{\Big(\varphi_1(x)}\varphi_p(x),x\varphi_p(x),\dots,x^{n_p-1}\varphi_p(x)\Big).\notag
\end{align}
We have
\begin{align}
&N!\int_{E^N}\Big(\prod_{i=1}^{N}e^{-V(x_i)}\,dx_i\Big)\,\det\big[\tilde{\psi}_i(x_j)\big]_{1\leq i,j\leq N}\,\det\big[\tilde{\varphi}_i(x_j)\big]_{1\leq i,j\leq N}\notag\\
&\phantom{N!\int_{E^N}}=(N!)^2\det\Bigg[\int_{E}\tilde{\psi}_i(x)\tilde{\varphi}_j(x)e^{-V(x)}dx\Bigg]_{1\leq i,j\leq N}\notag\\
&\phantom{N!\int_{E^N}}=\binom{N}{m_1,m_2,\dots,m_q}\binom{N}{n_1,n_2,\dots,n_p}\int_{E^N}\Big(\prod_{i=1}^{N}e^{-V(x_i)}\,dx_i\Big)\Big(\Delta_{m_1}(x^{(1)})\prod_{i=1}^{m_1}\psi_1(x_{i})\Big)\notag\\
&\phantom{N!\int_{E^N}=}\times\dots\,\times\Big(\Delta_{m_q}(x^{(q)})\prod_{i=1}^{m_q}\psi_q(x_{m_1+\dots+m_{q-1}+i})\Big)\notag\\
&\phantom{N!\int_{E^N}=}\times\sum_{\sigma\in S_N}(-1)^{\sigma}\Big[\Big(\Delta_{n_1}(x_{\sigma(1)},\dots,x_{\sigma(n_1)})\prod_{i=1}^{n_1}\varphi_1(x_{\sigma(i)})\Big)\notag\\
&\phantom{N!\int_{E^N}=\times\Big[}\times\dots\times\Big(\Delta_{n_p}(x_{\sigma(n_1+\dots+n_{p-1}+1)},\dots,x_{\sigma(n_1+\dots+n_p)})\prod_{i=1}^{n_p}\varphi_p(x_{\sigma(n_1+\dots+n_{p-1}+i)})\Big)\Big],\label{formula proposition 1}
\end{align}
where $x^{(1)}=(x_1,x_2,\dots,x_{m_1})$, \dots, $x^{(q)}=(x_{m_1+\dots+m_{q-1}+1},\dots,x_{m_1+\dots+m_q})$, and $\Delta_n(x_1,\dots,x_n)=\det[x_j^{i-1}]_{1\leq i,j\leq n}$ is the Vandermonde determinant.
\end{proposition}

\begin{proof}
Let $\tilde{P}(E;p,q)$ be left hand side in (\ref{formula proposition 1})
\begin{align}
\tilde{P}(E;p,q):=N!\int_{E^N}\Big(\prod_{i=1}^{N}e^{-V(x_i)}\,dx_i\Big)\,\det\big[\tilde{\psi}_i(x_j)\big]_{1\leq i,j\leq N}\,\det\big[\tilde{\varphi}_i(x_j)\big]_{1\leq i,j\leq N}.\notag
\end{align}
The first identity in (\ref{formula proposition 1}) is a consequence of applying the following standard identity
\begin{align}
\det[a_{ij}]_{1\leq i,j\leq N}\det[b_{ij}]_{1\leq i,j\leq N}=\sum_{\sigma\in S_N}\det\big[a_{i,\sigma(j)}b_{j,\sigma(j)}\big]_{1\leq i,j\leq N},\notag
\end{align}
and distributing the integration over the different columns.

We prove now the second identity in (\ref{formula proposition 1}). Working out the determinant $\det\big[\tilde{\psi}_i(x_j)\big]_{1\leq i,j\leq N}$ we obtain
\begin{align}
\tilde{P}(E;p,q)&=N!\sum_{\sigma\in S_N}(-1)^\sigma\int_{E^N}\Big(\prod_{i=1}^{N}e^{-V(x_i)}\,dx_i\Big)\,\Big(\prod_{i=1}^{N}\tilde{\psi}_i(x_{\sigma(i)})\Big)\,\det\big[\tilde{\varphi}_i(x_j)\big]_{1\leq i,j\leq N}.\notag
\end{align}
In each term of this summation, we make the change of variables $x_i=y_{\sigma^{-1}(i)}$, $1\leq i\leq N$. We have
\begin{align}
\tilde{P}(E;p,q)&=(N!)^2\int_{E^N}\Big(\prod_{i=1}^{N}e^{-V(y_i)}\,dy_i\Big)\,\Big(\prod_{i=1}^{N}\tilde{\psi}_i(y_{i})\Big)\,\det\big[\tilde{\varphi}_i(y_j)\big]_{1\leq i,j\leq N},\notag
\end{align}
since $\det\big[\tilde{\varphi}_i(y_j)\big]_{1\leq i,j\leq N}=(-1)^\sigma \det\big[\tilde{\varphi}_i(x_{j})\big]_{1\leq i,j\leq N}$. Now take $\sigma_1\in S_{m_1}$, $\sigma_2\in S_{m_2}$, \dots, $\sigma_q\in S_{m_q}$ arbitrarily and define the permutation $\sigma:=\sigma_1\times\sigma_2\times\dots\times\sigma_q\,\in\,S_N$. Consider the following change of variables $y\to z$ defined by $\sigma$:
\begin{align}
(y_1,\dots,y_{m_1})&=(z_{\sigma_1(1)},\dots,z_{\sigma_1(m_1)}),\notag\\
(y_{m_1+1},\dots,y_{m_1+m_2})&=(z_{m_1+\sigma_2(1)},\dots,z_{m_1+\sigma_2(m_2)}),\notag\\
&\vdots\notag\\
(y_{m_1+\dots+m_{q-1}+1},\dots,y_{m_1+\dots+m_q})&=(z_{m_1+\dots+m_{q-1}+\sigma_q(1)},\dots,z_{m_1+\dots+m_{q-1}+\sigma_q(m_q)}).\notag
\end{align}
This change of variables leaves the integral unchanged. Consequently, if we sum over all the permutations $\sigma_1\times\sigma_2\times\dots\times\sigma_q\,\in\,S_{m_1}\times\dots\times S_{m_q}$ and divide by $m_1!m_2!\dots m_q!$, we have by the definition of $\tilde{\psi}_i$ and $\Delta_n(z)$
\begin{gather}
\tilde{P}(E;p,q)=N!\binom{N}{m_1,m_2,\dots,m_q}\int_{E^N}\Big(\prod_{i=1}^{N}e^{-V(z_i)}\,dz_i\Big)\Big(\Delta_{m_1}(z^{(1)})\prod_{i=1}^{m_1}\psi_1(z_{i})\Big)\notag\\
\times\Big(\Delta_{m_2}(z^{(2)})\prod_{i=1}^{m_2}\psi_2(z_{m_1+i})\Big)\times\dots\times\Big(\Delta_{m_q}(z^{(q)})\prod_{i=1}^{m_q}\psi_q(z_{m_1+\dots+m_{q-1}+i})\Big)\det\big[\tilde{\varphi}_i(z_j)\big]_{1\leq i,j\leq N}.\label{P tilde}
\end{gather}
We develop the determinant $\det\big[\tilde{\varphi}_i(z_j)\big]_{1\leq i,j\leq N}$ in this expression
\begin{align}
\det\big[\tilde{\varphi}_i(z_j)\big]_{1\leq i,j\leq N}&=\sum_{\sigma\in S_N}(-1)^\sigma\prod_{i=1}^{N}\tilde\varphi_i\big(z_{\sigma(i)}\big)\notag\\
&=\sum_{\sigma\in S_N}(-1)^\sigma\Big(\prod_{i=1}^{n_1}\varphi_1\big(z_{\sigma(i)}\big)z_{\sigma(i)}^{i-1}\Big)\Big(\prod_{i=1}^{n_2}\varphi_2\big(z_{\sigma(n_1+i)}\big)z_{\sigma(n_1+i)}^{i-1}\Big)\times\dots\notag\\
&\phantom{=\sum_{\sigma\in S_N}}\times\Big(\prod_{i=1}^{n_p}\varphi_p\big(z_{\sigma(n_1+\dots+n_{p-1}+i)}\big)z_{\sigma(n_1+\dots+n_{p-1}+i)}^{i-1}\Big).\notag
\end{align}
Fix $\tilde{\sigma}_1\in S_{n_1},\dots,\tilde{\sigma}_p\in S_{n_p}$ and, for a given permutation $\sigma\in S_N$, let $\tilde{\sigma}\in S_N$ be such that $\sigma=\tilde\sigma\circ(\tilde\sigma_1\times\tilde\sigma_2\times\dots\times\tilde\sigma_p)$, but when $\tilde{\sigma}$ runs over $S_N$, then so does $\sigma$. Consequently we have
\begin{align}
\det\big[\tilde{\varphi}_i(z_j)\big]_{1\leq i,j\leq N}&=\sum_{\tilde\sigma\in S_N}(-1)^{\tilde\sigma}(-1)^{\tilde\sigma_1}\dots(-1)^{\tilde\sigma_p}\Big(\prod_{i=1}^{n_1}\varphi_1\big(z_{\tilde\sigma\circ\tilde\sigma_1(i)}\big)z_{\tilde\sigma\circ\tilde\sigma_1(i)}^{i-1}\Big)\times\dots\notag\\
&\phantom{=\sum_{\sigma\in S_N}}\times\Big(\prod_{i=1}^{n_p}\varphi_p\big(z_{\tilde\sigma(n_1+\dots+n_{p-1}+\tilde\sigma_p(i))}\big)z_{\tilde\sigma(n_1+\dots+n_{p-1}+\tilde\sigma_p(i))}^{i-1}\Big).\notag
\end{align}
We substitute this expression in equation (\ref{P tilde}). For each $\tilde\sigma\in S_N$ we further sum over $\tilde\sigma_1\in S_{n_1},\tilde\sigma_2\in S_{n_2},\dots,\tilde\sigma_p\in S_{n_p}$ and so must divide by $n_1!n_2!\dots n_p!$ as we have overcounted. We then obtain the second equality in (\ref{formula proposition 1}). This ends the proof.
\end{proof}


\section{An integrable deformation and $(p+q)$-component KP}

The connection of the problem of non-intersecting brownian motions on $\mathbb{R}$ with the multi-component KP hierarchy is explained in \cite{A.V.M.V.}. The main ideas are being sketched in this section.


\subsection{The $(p+q)$-component KP hierarchy}

Define two sets of weights
\begin{align}
\psi_1(x),\dots,\psi_q(x),\quad\text{and}\quad \varphi_1(y),\dots,\varphi_p(y),\qquad\text{with}\ x,y\in\mathbb{R},\notag
\end{align}
and deformed weights depending on time parameters $s^{(\alpha)}=\big(s^{(\alpha)}_1,s^{(\alpha)}_2,\dots\big)$, $1\leq \alpha\leq q$, and $t^{(\beta)}=\big(t^{(\beta)}_1,t^{(\beta)}_2,\dots\big)$, $1\leq \beta\leq p$, denoted by
\begin{align}
\psi_\alpha^{-s}(x):=\psi_\alpha(x)e^{-\sum_{k=1}^{\infty}s^{(\alpha)}_kx^k},\quad\text{and}\quad \varphi_\beta^{t}(y):=\varphi_\beta(y)e^{\sum_{k=1}^{\infty}t^{(\beta)}_ky^k}.\notag
\end{align}
For each set of integers
\begin{align}
\vec{m}=(m_1,\dots,m_q),\quad \vec{n}=(n_1,\dots,n_p),\quad \text{with}\ |\vec{m}|=|\vec{n}|,\notag
\end{align}
$|\vec{m}|=\sum_{i=1}^{q}m_i$, $|\vec{n}|=\sum_{j=1}^{p}n_j$, consider the determinant of the moment matrix $T_{\vec{m}\vec{n}}$ of size $|\vec{m}|=|\vec{n}|$, composed of $pq$ blocks of sizes $m_in_j$. The moments are taken with regard to a (not necessarily symmetric) inner product $\left\langle \cdot\,|\,\cdot\right\rangle$
\begin{align}
\tau_{\vec{m}\vec{n}}&(s^{(1)},\dots,s^{(q)};t^{(1)},\dots,t^{(p)})\notag\\
&:=\det T_{\vec{m}\vec{n}}\notag\\
&:=\det\left(\begin{array}{ccc}
\Big(\left\langle x^i\psi_1^{-s}(x)\,|\,y^j\varphi_1^{t}(y)\right\rangle\Big)_{\substack{0\leq i< m_1 \\ 0\leq j< n_1}} & \dots & \Big(\left\langle x^i\psi_1^{-s}(x)\,|\,y^j\varphi_p^{t}(y)\right\rangle\Big)_{\substack{0\leq i< m_1 \\ 0\leq j< n_p}} \\
 \\
\vdots & & \vdots \\
 \\
\Big(\left\langle x^i\psi_q^{-s}(x)\,|\,y^j\varphi_1^{t}(y)\right\rangle\Big)_{\substack{0\leq i< m_q \\ 0\leq j< n_1}} & \dots & \Big(\left\langle x^i\psi_q^{-s}(x)\,|\,y^j\varphi_p^{t}(y)\right\rangle\Big)_{\substack{0\leq i< m_q \\ 0\leq j< n_p}} 
\end{array}\right).\label{determinant moment matrix}
\end{align}

\begin{theorem}[Adler, van Moerbeke, Vanhaecke ~\cite{A.V.M.V.}]\label{p+q KP}
The determinants of the block matrices $\tau_{\vec{m}\vec{n}}$ satisfy the bilinear relations\footnote{Introduce the notation $[z]=(z,\frac{z^2}{2},\frac{z^3}{3},\dots)$ for $z\in\mathbb{C}$, and let $\vec{e}_1=(1,0,0,\dots)$, $\vec{e}_2=(0,1,0,\dots)$, \dots. Only shifted times will be made explicit in the functions $\tau_{\vec{m}\vec{n}}$. The integrals are contour integrals along a small circle about $\infty$, with formal Laurent series as integrand.}
\begin{align}
\sum_{\beta=1}^{p}\oint_{\infty}(-1)^{\sigma_\beta(\vec{n})}\tau_{\vec{m},\vec{n}-\vec{e}_\beta}(t^{(\beta)}-[z^{-1}])\tau_{\vec{m}^*,\vec{n}^*+\vec{e}_\beta}(t^{(\beta)*}+[z^{-1}])e^{\sum_{k=1}^{\infty}(t^{(\beta)}_k-t^{(\beta)*}_k)z^k}z^{n_\beta-n^*_{\beta}-2}dz\notag\\
=\sum_{\alpha=1}^{q}\oint_{\infty}(-1)^{\sigma_\alpha(\vec{m})}\tau_{\vec{m}+\vec{e}_\alpha,\vec{n}}(s^{(\alpha)}-[z^{-1}])\tau_{\vec{m}^*-\vec{e}_\alpha,\vec{n}^*}(s^{(\alpha)*}+[z^{-1}])e^{\sum_{k=1}^{\infty}(s^{\alpha}_k-s^{(\alpha)*}_k)z^k}z^{m^*_\alpha-m_{\alpha}-2}dz,\notag
\end{align}
for all $\vec{m},\vec{n},\vec{m}^*,\vec{n}^*$ such that $|\vec{m}^*|=|\vec{n}^*|+1$ and $|\vec{m}|=|\vec{n}|-1$, and all $s,t,s^*,t^*\in\mathbb{C}^\infty$, and where
\begin{align}
\sigma_\alpha(\vec{m})=\sum_{\alpha'=1}^{\alpha}(m_{\alpha'}-m^*_{\alpha'}),\quad\text{and}\quad \sigma_\beta(\vec{n})=\sum_{\beta'=1}^{\beta}(n_{\beta'}-n^*_{\beta'}).\notag
\end{align}
\end{theorem}
These identities define the $(p+q)$-component KP hierarchy, as described by Ueno and Takasaki \cite{U.T.}. \\

Define the Hirota symbol between functions $f=f(t_1,t_2,\dots)$ and $g=g(t_1,t_2,\dots)$, given a polynomial $p(t_1,t_2,\dots)$, namely
\begin{align}
p\Big(\frac{\partial}{\partial t_1},\frac{\partial}{\partial t_2},\dots\Big)f\circ g:=p\Big(\frac{\partial}{\partial y_1},\frac{\partial}{\partial y_2},\dots\Big)f(t+y)g(t-y)\Big|_{y=0}.\notag
\end{align}
This operation extends readily to the case where $p(t_1,t_2,\dots)$ is a Taylor series in $t_1,t_2,\dots$. We also need the elementary Schur polynomials $s_l$, which are defined by $e^{\sum_{k=1}^{\infty}t_kz^k}:=\sum_{k=0}^{\infty}s_k(t)z^k$, for $l\geq 0$, and $s_l(t)=0$ for $l<0$. Moreover, set
\begin{align}
s_l(\tilde{\partial}_t):=s_l\Big(\frac{\partial}{\partial t_1},\frac{1}{2}\frac{\partial}{\partial t_2},\frac{1}{3}\frac{\partial}{\partial t_3},\dots\Big).\notag
\end{align}
With these notations, computing the residues about $z=\infty$ in the contour integrals above, the functions $\tau_{\vec{m}\vec{n}}$, with $|\vec{m}|=|\vec{n}|$, are found to satisfy the following PDE's :
\begin{align}
&\tau_{\vec{m}\vec{n}}^2\frac{\partial^2}{\partial t^{(\beta)}_{j+1}\partial t^{(\beta')}_{1}}\log\tau_{\vec{m}\vec{n}}=s_{j+2\delta_{\beta\beta'}}(\tilde\partial_{t^{(\beta)}})\tau_{\vec{m},\vec{n}+\vec{e}_\beta-\vec{e}_{\beta'}}\circ\tau_{\vec{m},\vec{n}-\vec{e}_\beta+\vec{e}_{\beta'}},\notag\\
&\tau_{\vec{m}\vec{n}}^2\frac{\partial^2}{\partial s^{(\alpha)}_{j+1}\partial s^{(\alpha')}_{1}}\log\tau_{\vec{m}\vec{n}}=s_{j+2\delta_{\alpha\alpha'}}(\tilde\partial_{s^{(\alpha)}})\tau_{\vec{m}-\vec{e}_\alpha+\vec{e}_{\alpha'},\vec{n}}\circ\tau_{\vec{m}+\vec{e}_\alpha-\vec{e}_{\alpha'},\vec{n}},\notag\\
&\tau_{\vec{m}\vec{n}}^2\frac{\partial^2}{\partial s^{(\alpha)}_{1}\partial t^{(\beta)}_{j+1}}\log\tau_{\vec{m}\vec{n}}=-s_{j}(\tilde\partial_{t^{(\beta)}})\tau_{\vec{m}+\vec{e}_\alpha,\vec{n}+\vec{e}_\beta}\circ\tau_{\vec{m}-\vec{e}_\alpha,\vec{n}-\vec{e}_\beta},\notag\\
&\tau_{\vec{m}\vec{n}}^2\frac{\partial^2}{\partial t^{(\beta)}_{1}\partial s^{(\alpha)}_{j+1}}\log\tau_{\vec{m}\vec{n}}=-s_{j}(\tilde\partial_{s^{(\alpha)}})\tau_{\vec{m}-\vec{e}_\alpha,\vec{n}-\vec{e}_\beta}\circ\tau_{\vec{m}+\vec{e}_\alpha,\vec{n}+\vec{e}_\beta}.\label{bilinear identities residue}
\end{align}


\subsection{An integrable deformation of the joint probability density function for the problem of non-intersecting Brownian motions}

We will now deform $P_{p,q}(E;a,b)$ defined in (\ref{normalized problem}) by adding extra time variables
\begin{align}
t^{(1)}=\big(t_1^{(1)},t_2^{(1)},\dots\big),\quad t^{(2)}=\big(t_1^{(2)},t_2^{(2)},\dots\big),\quad\dots\quad,\quad t^{(p)}=\big(t_1^{(p)},t_2^{(p)},\dots\big),\notag\\
s^{(1)}=\big(s_1^{(1)},s_2^{(1)},\dots\big),\quad s^{(2)}=\big(s_1^{(2)},s_2^{(2)},\dots\big),\quad\dots\quad,\quad s^{(q)}=\big(s_1^{(q)},s_2^{(q)},\dots\big),\notag
\end{align}
and auxiliary variables
\begin{align}
(\alpha_1,\dots,\alpha_q),\quad (\beta_1,\dots,\beta_p),\notag
\end{align}
such that $\sum_{i=1}^{q}\alpha_i=\sum_{j=1}^{p}\beta_j=0$. First set
\begin{align}
&\psi_i^{-s}(x):=e^{a_ix+\alpha_ix^2-\sum_{j=1}^{\infty}s_j^{(i)}x^j},\quad1\leq i\leq q,\notag\\
&\varphi_i^{t}(x):=e^{b_ix+\beta_ix^2+\sum_{j=1}^{\infty}t_j^{(i)}x^j},\quad1\leq i\leq p.\notag
\end{align}
We define
\begin{align}
P_{p,q}\big(E;a,b;(t,s),(\alpha,\beta)\big)&=\frac{\tau_{\vec{m},\vec{n}}^E(t,s;\alpha,\beta;a,b)}{\tau_{\vec{m},\vec{n}}^\mathbb{R}(t,s;\alpha,\beta;a,b)},\label{deformed normalized problem}
\end{align}
with
\begin{align}
\tau_{\vec{m},\vec{n}}^E(t,s;\alpha,\beta;a,b):=\frac{1}{N!}\int_{E^N}\Big(\prod_{i=1}^{N}e^{\frac{-x_i^2}{2}}\,dx_i\Big)\,\det\big[\tilde{\psi}_i^{-s}(x_j)\big]_{1\leq i,j\leq N}\,\det\big[\tilde{\varphi}_i^t(x_j)\big]_{1\leq i,j\leq N},\label{tau}
\end{align}
\sloppy{where $(t,s)=\big(t^{(1)},\dots,t^{(p)};s^{(1)},\dots,s^{(q)}\big)$, $(\alpha,\beta)=(\alpha_1,\dots,\alpha_{q-1};\beta_1,\dots,\beta_{p-1})$, $(a,b)=(a_1,\dots,a_{q-1};b_1,\dots,b_{p-1})$, and}
\begin{align}
\big(\tilde{\psi}^{-s}_1(x),\dots,\tilde{\psi}^{-s}_N(x)\big):=&\Big(\psi^{-s}_1(x),x\psi^{-s}_1(x),\dots,x^{m_1-1}\psi^{-s}_1(x),\dots,\notag\\
&\phantom{\Big(\psi^{-s}_1(x)}\psi^{-s}_q(x),x\psi^{-s}_q(x),\dots,x^{m_q-1}\psi^{-s}_q(x)\Big),\notag\\
\big(\tilde{\varphi}^t_1(x),\dots,\tilde{\varphi}^t_N(x)\big):=&\Big(\varphi^{t}_1(x),x\varphi^{t}_1(x),\dots,x^{n_1-1}\varphi^{t}_1(x),\dots,\notag\\
&\phantom{\Big(\varphi^{t}_1(x)}\varphi^{t}_p(x),x\varphi^{t}_p(x),\dots,x^{n_p-1}\varphi^{t}_p(x)\Big).\notag
\end{align}
Observe that $P_{p,q}(E;a,b)=P_{p,q}\big(E;a,b;(t,s),(\alpha,\beta)\big)\big|_{\mathcal{L}}$, where $\mathcal{L}=\{(t,s)=(0,0),\alpha=\beta=0\}$. By virtue of proposition \ref{proposition 1} we have
\begin{align}
\tau_{\vec{m},\vec{n}}^E(t&,s;\alpha,\beta;a,b)=\det\Bigg[\int_{E}\tilde{\psi}^{-s}_i(x)\tilde{\varphi}^t_j(x)e^{\frac{-x^2}{2}}dx\Bigg]_{1\leq i,j\leq N}\notag\\
&=\det\left(\begin{array}{ccc}
\Big(\int_{E}x^{i+j}\psi_1^{-s}\varphi_1^{t}e^{-\frac{x^2}{2}}dx\Big)_{\substack{0\leq i< m_1 \\ 0\leq j< n_1}} & \dots & \Big(\int_{E}x^{i+j}\psi_1^{-s}\varphi_p^{t}e^{-\frac{x^2}{2}}dx\Big)_{\substack{0\leq i< m_1 \\ 0\leq j< n_p}} \\
 \\
\vdots & & \vdots \\
 \\
\Big(\int_{E} x^{i+j}\psi_q^{-s}\varphi_1^{t}e^{-\frac{x^2}{2}}dx\Big)_{\substack{0\leq i< m_q \\ 0\leq j< n_1}} & \dots & \Big(\int_{E}x^{i+j}\psi_q^{-s}\varphi_p^{t}e^{-\frac{x^2}{2}}dx\Big)_{\substack{0\leq i< m_q \\ 0\leq j< n_p}} 
\end{array}\right),\label{deformed probability as a determinant}
\end{align}
where for simplicity we have left out the dependence of $\varphi_i^{-s}$ and $\psi_j^{t}$ on $x$. We have also
\begin{align}
\tau_{\vec{m},\vec{n}}^E(t,s;\alpha,\beta;a,b)=&\frac{1}{\prod_{i=1}^{q}m_i!\prod_{j=1}^{p}n_j!}\int_{E^N}\Big(\Delta_{m_1}(x^{(1)})\prod_{i=1}^{m_1}\psi^{-s}_1(x_{i})e^{\frac{-x_i^2}{2}}\,dx_i\Big)\times\dots\,\times\notag\\
&\Big(\Delta_{m_q}(x^{(q)})\prod_{i=m_1+\dots+m_{q-1}+1}^{m_1+\dots+m_q}\psi^{-s}_q(x_{i})e^{\frac{-x_i^2}{2}}\,dx_i\Big)\notag\\
&\times\sum_{\sigma\in S_N}(-1)^{\sigma}\Big[\Big(\Delta_{n_1}(x_{\sigma(1)},\dots,x_{\sigma(n_1)})\prod_{i=1}^{n_1}\varphi^{t}_1(x_{\sigma(i)})\Big)\times\dots\times\notag\\
&\Big(\Delta_{n_p}(x_{\sigma(n_1+\dots+n_{p-1}+1)},\dots,x_{\sigma(n_1+\dots+n_p)})\prod_{i=n_1+\dots+n_{p-1}+1}^{n_1+\dots+n_p}\varphi^{t}_p(x_{\sigma(i)})\Big)\Big].\label{deformed probability as a integral with vandermonde determinants}
\end{align}

The determinant of the moment matrix (\ref{determinant moment matrix}) with regard to the inner product $\left\langle f,g\right\rangle=\int_Ef(z)g(z)e^{-z^2/2}dz$, with
\begin{align}
&\psi_i(x):=e^{a_ix+\alpha_ix^2},\quad1\leq i\leq q,&\varphi_j(x):=e^{b_jx+\beta_jx^2},\quad1\leq j\leq p,\notag
\end{align}
is the same as the determinant (\ref{deformed probability as a determinant}). Therefore, by virtue of Theorem \ref{p+q KP}, $\tau_{\vec{m},\vec{n}}^E(t,s;\alpha,\beta;a,b)$ satisfies the $(p+q)$-component KP hierarchy.

\begin{corollary}\label{corollary bilinear identities}
The function $\tau_{\vec{m},\vec{n}}^E(t,s;\alpha,\beta;a,b)$ satisfies the following identities, $1\leq k,k'\leq q$, $1\leq l,l'\leq p$, $k\neq k'$, $l\neq l'$,
\begin{align}
&\frac{\partial}{\partial t^{(l)}_1}\,\ln\frac{\tau_{\vec{m},\vec{n}+\vec{e}_l-\vec{e}_{l'}}^E}{\tau_{\vec{m},\vec{n}-\vec{e}_l+\vec{e}_{l'}}^E}=\frac{\frac{\partial^2}{\partial t^{(l)}_2\partial t^{(l')}_1}\,\ln\tau_{\vec{m},\vec{n}}^E}{\frac{\partial^2}{\partial t^{(l)}_1\partial t^{(l')}_1}\,\ln\tau_{\vec{m},\vec{n}}^E},&\frac{\partial}{\partial s^{(k)}_1}\,\ln\frac{\tau_{\vec{m}-\vec{e}_k+\vec{e}_{k'},\vec{n}}^E}{\tau_{\vec{m}+\vec{e}_k-\vec{e}_{k'},\vec{n}}^E}=\frac{\frac{\partial^2}{\partial s^{(k)}_2\partial s^{(k')}_1}\,\ln\tau_{\vec{m},\vec{n}}^E}{\frac{\partial^2}{\partial s^{(k)}_1\partial s^{(k')}_1}\,\ln\tau_{\vec{m},\vec{n}}^E},\notag\\
&\frac{\partial}{\partial t^{(l)}_1}\,\ln\frac{\tau_{\vec{m}+\vec{e}_k,\vec{n}+\vec{e}_l}^E}{\tau_{\vec{m}-\vec{e}_k,\vec{n}-\vec{e}_l}^E}=\frac{\frac{\partial^2}{\partial t^{(l)}_2\partial s^{(k)}_1}\,\ln\tau_{\vec{m},\vec{n}}^E}{\frac{\partial^2}{\partial t^{(l)}_1\partial s^{(k)}_1}\,\ln\tau_{\vec{m},\vec{n}}^E},&\frac{\partial}{\partial s^{(k)}_1}\,\ln\frac{\tau_{\vec{m}-\vec{e}_k,\vec{n}-\vec{e}_l}^E}{\tau_{\vec{m}+\vec{e}_k,\vec{n}+\vec{e}_l}^E}=\frac{\frac{\partial^2}{\partial s^{(k)}_2\partial t^{(l)}_1}\,\ln\tau_{\vec{m},\vec{n}}^E}{\frac{\partial^2}{\partial s^{(k)}_1\partial t^{(l)}_1}\,\ln\tau_{\vec{m},\vec{n}}^E}.\label{bilinear}
\end{align}
\end{corollary}

\begin{proof}
We shall only give the proof of the first identity. The two first elementary Schur polynomials are given by
\begin{align*}
s_0(x_1,x_2,\dots)=1,\qquad s_1(x_1,x_2,\dots)=x_1.
\end{align*}
Consequently, the first equation in \eqref{bilinear identities residue} with $j=0$ and $l\neq l'$ gives
\begin{align}
\big(\tau_{\vec{m}\vec{n}}^E\big)^2\frac{\partial^2}{\partial t^{(l)}_{1}\partial t^{(l')}_{1}}\log\tau_{\vec{m}\vec{n}}^E&=s_{0}(\tilde\partial_{t^{(l)}})\tau_{\vec{m},\vec{n}+\vec{e}_l-\vec{e}_{l'}}^E\circ\tau_{\vec{m},\vec{n}-\vec{e}_l+\vec{e}_{l'}}^E=\tau_{\vec{m},\vec{n}+\vec{e}_l-\vec{e}_{l'}}^E\,\tau_{\vec{m},\vec{n}-\vec{e}_l+\vec{e}_{l'}}^E,\label{useful formula multi-kp 1}
\end{align}
while for $j=1$ and $l\neq l'$ it gives
\begin{align}
\big(\tau_{\vec{m}\vec{n}}^E\big)^2\frac{\partial^2}{\partial t^{(l)}_{2}\partial t^{(l')}_{1}}\log\tau_{\vec{m}\vec{n}}^E&=s_{1}(\tilde\partial_{t^{(l)}})\tau_{\vec{m},\vec{n}+\vec{e}_l-\vec{e}_{l'}}^E\circ\tau_{\vec{m},\vec{n}-\vec{e}_l+\vec{e}_{l'}}^E\notag\\
&=\tau_{\vec{m},\vec{n}-\vec{e}_l+\vec{e}_{l'}}^E\frac{\partial}{\partial t_1^{(l)}}\tau_{\vec{m},\vec{n}+\vec{e}_l-\vec{e}_{l'}}^E-\tau_{\vec{m},\vec{n}+\vec{e}_l-\vec{e}_{l'}}^E\frac{\partial}{\partial t_1^{(l)}}\tau_{\vec{m},\vec{n}-\vec{e}_l+\vec{e}_{l'}}^E.\label{useful formula multi-kp 2}
\end{align}
Taking the ratio of \eqref{useful formula multi-kp 2} and \eqref{useful formula multi-kp 1} yields the first formula of Corollary \ref{corollary bilinear identities}. The other identities are obtained in a similar way.
\end{proof}


\section{Virasoro constraints}

Let us introduce the following differential operators
\begin{align}
\mathbb{J}^{(1)}_{m,k}(t)=&\frac{\partial}{\partial t_m}+(-m)t_{-m}+k\,\delta_{0,m},\notag\\
\mathbb{J}^{(2)}_{m,k}(t)=&\frac{1}{2}\Big(\sum_{i+j=m}\frac{\partial^2}{\partial t_i\partial t_j}+2\sum_{i\geq 1}it_i\frac{\partial}{\partial t_{i+m}}+\sum_{i+j=-m}it_ijt_j\Big)\notag\\
&\qquad\qquad+\Big(k+\frac{m+1}{2}\Big)\Big(\frac{\partial}{\partial t_m}+(-m)t_{-m}\Big)+\frac{k(k+1)}{2}\,\delta_{m,0}.\notag
\end{align}
Those operators satisfy the Heisenberg and Virasoro algebra respectively 
\begin{align}
&\big[\mathbb{J}^{(1)}_{k,n}(t),\mathbb{J}^{(1)}_{l,n}(t)\big]=k\,\delta_{k,-l},&\big[\mathbb{J}^{(2)}_{k,n}(t),\mathbb{J}^{(2)}_{l,n}(t)\big]=(k-l)\mathbb{J}^{(2)}_{k+l,n}-\Big(\frac{k^3-k}{6}\Big)\,\delta_{k,-l},\notag
\end{align}
and interact as follows
\begin{align}
\big[\mathbb{J}^{(2)}_{k,n}(t),\mathbb{J}^{(1)}_{l,n}(t)\big]=-l\,\mathbb{J}^{(1)}_{k+l,n}(t)+\frac{k(k+1)}{2}\delta_{k,-l}.\notag
\end{align}

We have the following lemma, proven by Adler and van Moerbeke ~\cite{A.V.M.3}.
\begin{lemma}[Adler-van Moerbeke ~\cite{A.V.M.3}]\label{simple variationnal equation}
Given $\rho(z)=e^{-V(z)}$, with
\begin{align}
-\frac{\rho'(z)}{\rho(z)}=V'(z)=\frac{g(z)}{f(z)}=\frac{\sum_{i=0}^{\infty}\nu_iz^i}{\sum_{i=0}^{\infty}\mu_iz^i},\notag
\end{align}
the integrand
\begin{align}
dI_N(z;t):=\Delta_N(z)\,\prod_{k=1}^{N}\Big(e^{\sum_{i=1}^{\infty}t_iz_k^i}\rho(z_k)dz_k\Big)\notag
\end{align}
satisfies the variational formula
\begin{align}
\frac{d}{d\epsilon}dI_N\big(z_i\mapsto z_i+\epsilon f(z_i)z_i^{k+1};t\big)\Big|_{\epsilon=0}=\sum_{l=0}^{\infty}\Big(\mu_l\,\mathbb{J}_{k+l,N}^{(2)}(t)-\nu_l\,\mathbb{J}_{k+l+1,N}^{(1)}(t)\Big)dI_N(z;t),\notag
\end{align}
for each $k\geq-1$. The contribution of the factor $\prod_{i=1}^{N}dz_i$ in this equation is 
\begin{align}
\sum_{l=0}^{\infty}\mu_l\,(l+k+1)\,\mathbb{J}_{k+l,N}^{(1)}\,dI_N(z;t).\notag
\end{align}
\end{lemma}

We define, for a given permutation $\sigma\in S_n$, the integrands
\begin{align}
dI_{\vec{m},\vec{n}}^\sigma\big(x;(t,s)\big)=&\Big(\Delta_{m_1}(x^{(1)})\prod_{i=1}^{m_1}\psi^{-s^{(1)}}_1(x_{i})e^{\frac{-x_i^2}{2}}\,dx_i\Big)\times\dots\,\times\notag\\
&\Big(\Delta_{m_q}(x^{(q)})\prod_{i=m_1+\dots+m_{q-1}+1}^{m_1+\dots+m_q}\psi^{-s^{(q)}}_q(x_{i})e^{\frac{-x_i^2}{2}}\,dx_i\Big)\notag\\
&\times\Big(\Delta_{n_1}(x_{\sigma(1)},\dots,x_{\sigma(n_1)})\prod_{i=1}^{n_1}\varphi^{t^{(1)}}_1(x_{\sigma(i)})\Big)\times\dots\times\notag\\
&\Big(\Delta_{n_p}(x_{\sigma(n_1+\dots+n_{p-1}+1)},\dots,x_{\sigma(n_1+\dots+n_p)})\prod_{i=n_1+\dots+n_{p-1}+1}^{n_1+\dots+n_p}\varphi^{t^{(p)}}_p(x_{\sigma(i)})\Big).\label{integrand}
\end{align}
We are looking for a variational equation for 
\begin{align}
dI_{\vec{m},\vec{n}}^\sigma\Big(x_i\mapsto x_i+\epsilon x_i^{k+1};(t,s)\Big).\notag
\end{align}
We have the following lemma.
\begin{lemma}\label{complete variationnal equation}
The integrand $dI_{\vec{m},\vec{n}}^\sigma\big(x;(t,s)\big)$ as defined in (\ref{integrand}), satisfies the following variational equation for each $\sigma\in S_N$ and $k\geq-1$
\begin{align}
\frac{d}{d\epsilon}dI_{\vec{m},\vec{n}}^\sigma\Big(x_i\mapsto x_i+\epsilon x_i^{k+1};(t,s)\Big)\Big|_{\epsilon=0}=\mathbb{V}_k^{\vec{m},\vec{n}}\big(dI_{\vec{m},\vec{n}}^\sigma\big),\notag
\end{align}
with
\begin{align}
\mathbb{V}_k^{\vec{m},\vec{n}}:=&\sum_{i=1}^{q}\Big[\mathbb{J}_{k,m_i}^{(2)}(-s^{(i)})+a_i\,\mathbb{J}_{k+1,m_i}^{(1)}(-s^{(i)})-\big(1-2\alpha_i\big)\,\mathbb{J}_{k+2,m_i}^{(1)}(-s^{(i)})\Big]\notag\\
&+\sum_{i=1}^{p}\Big[\mathbb{J}_{k,n_i}^{(2)}(t^{(i)})+b_i\,\mathbb{J}_{k+1,n_i}^{(1)}(t^{(i)})+2\beta_i\,\mathbb{J}_{k+2,n_i}^{(1)}(t^{(i)})-(k+1)\,\mathbb{J}_{k,n_i}^{(1)}(t^{(i)})\Big]
.\label{Virasoro operator}
\end{align}
\end{lemma}
\begin{proof}
By the Leibniz rule, applying Lemma \ref{simple variationnal equation} to each factor in \eqref{integrand} and adding all these contributions yields \eqref{Virasoro operator}.
\end{proof}
As the variational formula in Lemma \ref{simple variationnal equation} is independent of the labeling of the variables in the integrand, it is a trivial but very important fact that the operator $\mathbb{V}_k^{\vec{m},\vec{n}}$ as defined in (\ref{Virasoro operator}) is independent of the choice of $\sigma\in S_n$. As a consequence, we have the following theorem.
\begin{theorem}\label{Virasoro constraints}
The function $\tau_{\vec{m},\vec{n}}^E(t,s)$ as defined in (\ref{tau}) satisfies the following Virasoro constraints
\begin{align}
\mathcal{B}_k\,\tau_{\vec{m},\vec{n}}^E=\mathbb{V}_k^{\vec{m},\vec{n}}\,\tau_{\vec{m},\vec{n}}^E,\qquad k\geq-1,\label{virasoro constraint}
\end{align}
with
\begin{align}
\mathcal{B}_k=\sum_{i=1}^{2r}c_i^{k+1}\,\frac{\partial}{\partial c_i},\notag
\end{align}
for $E=\cup_{i=1}^{r}[c_{2i-1},c_{2i}]\subset\mathbb{R}$.
\end{theorem}

\begin{proof}
By virtue of formula \eqref{deformed probability as a integral with vandermonde determinants} expressing $\tau_{\vec{m},\vec{n}}^E(t,s;\alpha,\beta;a,b)$ as a $N$-uple integral over $E$, we have
\begin{align*}
\tau_{\vec{m},\vec{n}}^E(t,s;\alpha,\beta;a,b)=\frac{1}{\prod_{i=1}^{q}m_i!\prod_{j=1}^{p}n_j!}\sum_{\sigma\in S_N}(-1)^\sigma \tau_{\vec{m},\vec{n}}^{E,\sigma}(t,s;\alpha,\beta;a,b),
\end{align*}
where $\tau_{\vec{m},\vec{n}}^{E,\sigma}(t,s;\alpha,\beta;a,b)$ is defined by
\begin{align*}
\tau_{\vec{m},\vec{n}}^{E,\sigma}(t,s;\alpha,\beta;a,b):=\int_{E^N}dI_{\vec{m},\vec{n}}^\sigma(x;(t,s)),
\end{align*}
with $dI_{\vec{m},\vec{n}}^\sigma(x;(t,s))$ as in \eqref{integrand}. For a fixed permutation $\sigma\in S_N$ and $k\geq -1$, we apply the change of variables $x_i\mapsto x_i+\epsilon x_i^{k+1}$, $1\leq i\leq N$, given in lemma \ref{complete variationnal equation}, in the integral defining $\tau_{\vec{m},\vec{n}}^{E,\sigma}(t,s;\alpha,\beta;a,b)$. This change of variables leaves the integral invariant, but induces a change of limits of integration, given by the inverse map
\begin{align*}
c_i\mapsto c_i-\epsilon c_i^{k+1}+O(\epsilon^2),\quad 1\leq i\leq 2r,
\end{align*}
for $\epsilon$ small enough. Consequently, differentiating the result with respect to $\epsilon$ and evaluating it at $\epsilon=0$, using the fundamental theorem of integral calculus together with Lemma \ref{complete variationnal equation}, we obtain
\begin{align}
\mathcal{B}_k\,\tau_{\vec{m},\vec{n}}^{E,\sigma}=\mathbb{V}_k^{\vec{m},\vec{n}}\,\tau_{\vec{m},\vec{n}}^{E,\sigma},\qquad k\geq-1,\label{Virasoro sigma}
\end{align}
with 
\begin{align}
\mathcal{B}_k=\sum_{i=1}^{2r}c_i^{k+1}\,\frac{\partial}{\partial c_i}.\notag
\end{align}
As noticed earlier, the operator $\mathbb{V}_k^{\vec{m},\vec{n}}$ does not depend on $\sigma\in S_N$. Consequently, summing \eqref{Virasoro sigma} over $\sigma\in S_N$ and dividing by $\prod_{i=1}^{q}m_i!\prod_{j=1}^{p}n_j!$, we obtain
\begin{align*}
\mathcal{B}_k\,\tau_{\vec{m},\vec{n}}^E=\mathbb{V}_k^{\vec{m},\vec{n}}\,\tau_{\vec{m},\vec{n}}^E,\qquad k\geq-1.
\end{align*}
This concludes the proof.
\end{proof}

When specializing the differential equations (\ref{virasoro constraint}) to $k=-1$ and $k=0$, we find that the tau-function $\tau_{\vec{m},\vec{n}}^E$ satisfies respectively
\begin{align}
&\mathcal{B}_{-1}\,\tau=\sum_{i\geq 2}\Bigg(\sum_{l=1}^{q}is^{(l)}_i\frac{\partial}{\partial s^{(l)}_{i-1}}+\sum_{l=1}^{p}it^{(l)}_i\frac{\partial}{\partial t^{(l)}_{i-1}}\Bigg)\tau+\sum_{l=1}^{q}(1-2\alpha_l)\frac{\partial\tau}{\partial s^{(l)}_{1}}+2\sum_{l=1}^{p}\beta_l\frac{\partial\tau}{\partial t^{(l)}_{1}}\notag\\
&\phantom{\mathcal{B}_{-1}\,\tau=}+\Big(\sum_{l=1}^{p}n_lt^{(l)}_{1}-\sum_{l=1}^{q}m_ls^{(l)}_{1}\Big)\tau+\Big(\sum_{l=1}^{q}a_lm_l+\sum_{l=1}^{p}b_ln_l\Big)\tau,\notag\\
&\mathcal{B}_{0}\,\tau=\sum_{i\geq 1}\Bigg(\sum_{l=1}^{q}is^{(l)}_i\frac{\partial}{\partial s^{(l)}_{i}}+\sum_{l=1}^{p}it^{(l)}_i\frac{\partial}{\partial t^{(l)}_{i}}\Bigg)\tau-\sum_{l=1}^{q}a_l\frac{\partial\tau}{\partial s^{(l)}_{1}}+\sum_{l=1}^{p}b_l\frac{\partial\tau}{\partial t^{(l)}_{1}}\notag\\
&\phantom{\mathcal{B}_{0}\,\tau=}+\sum_{l=1}^{q}(1-2\alpha_l)\frac{\partial\tau}{\partial s^{(l)}_{2}}+2\sum_{l=1}^{p}\beta_l\frac{\partial\tau}{\partial t^{(l)}_{2}}+\frac{1}{2}\Big(\sum_{l=1}^{q}m_l^2+\sum_{l=1}^{p}n_l^2\Big)\tau.\label{Virasoro -1 and 0}
\end{align}

The Virasoro constraints (\ref{virasoro constraint}) play a very crucial role in finding a PDE for the function $\log\mathbb{P}_{b_1,\dots,b_p}^{a_1,\dots,a_q}\big(all\ x_i(t)\in E\big)$ in the variables $a_1,\dots, a_q$, $b_1,\dots,b_p$ and the endpoints of the set $E$. In the next section, we will prove the existence of a PDE for the logarithm of the normalized problem $P_{p,q}(E;a,b)$ defined in \eqref{normalized problem}, and deduce Theorem \ref{Theorem introduction} from it. The normalized problem is related to the function $\tau_{\vec{m},\vec{n}}^E(t,s;\alpha,\beta;a,b)$ on the locus $\mathcal{L}=\{(t,s)=0,(\alpha,\beta)=0\}$ through formula \eqref{deformed normalized problem}. As we have seen, this function is a tau-function of the $(p+q)$-component KP hierarchy and thus satisfies the PDE's \eqref{useful formula multi-kp 2}. As the Virasoro constraints involve derivatives with respect to the endpoints of the set $E$, as well as derivatives with respect to the time variables $(t,s)$, we will prove that they can be used to eliminate all the derivatives with respect to the time variables in \eqref{useful formula multi-kp 2} on the locus $\mathcal{L}$. The proof proceeds in two main steps. In the first step, the Virasoro constraints, together with the linear conditions imposed on $a_i,\alpha_i,b_j,\beta_j$, $1\leq i\leq q$ and $1\leq j\leq p$
\begin{align}
\sum_{i=1}^{q}a_i=\sum_{i=1}^{p}b_i=\sum_{i=1}^{q}\alpha_i=\sum_{i=1}^{p}\beta_i=0,\label{linear conditions}
\end{align}
are used to express on the locus $\mathcal{K}=\{(t,s)=0\}$ all the derivatives with respect to the time variables in \eqref{useful formula multi-kp 2} in terms of derivatives with respect to the auxiliary variables $\alpha_1,\dots,\alpha_q$ and $\beta_1,\dots,\beta_p$. In the second step, using a combinatorial argument, it will be shown that, on the locus $\mathcal{L}$, all these derivatives with respect to $\alpha_1,\dots,\alpha_q$ and $\beta_1,\dots,\beta_p$ can be eliminated. Both steps will be performed in the next section. We end this section with some consequences of Theorem \ref{Virasoro constraints}.

From the linear conditions \eqref{linear conditions} it follows that the function $\tau_{\vec{m},\vec{n}}^E(t,s;\alpha,\beta;a,b)$ as defined in (\ref{tau}) satisfies the following equations
\begin{align}
\sum_{l=1}^{q}\frac{\partial\tau}{\partial s_i^{(l)}}+\sum_{l=1}^{p}\frac{\partial\tau}{\partial t_i^{(l)}}=0,\quad i\geq 1,\label{interesting 1}
\end{align}
and
\begin{align}
&\frac{\partial\tau}{\partial a_i}=-\frac{\partial\tau}{\partial s^{(i)}_1}+\frac{\partial\tau}{\partial s^{(q)}_1},\quad 1\leq i\leq q-1,\label{interesting 2}\\
&\frac{\partial\tau}{\partial b_i}=\frac{\partial\tau}{\partial t^{(i)}_1}-\frac{\partial\tau}{\partial t^{(p)}_1},\quad 1\leq i\leq p-1,\label{interesting 3}\\
&\frac{\partial\tau}{\partial \alpha_i}=-\frac{\partial\tau}{\partial s^{(i)}_2}+\frac{\partial\tau}{\partial s^{(q)}_2},\quad 1\leq i\leq q-1,\label{interesting 4}\\
&\frac{\partial\tau}{\partial \beta_i}=\frac{\partial\tau}{\partial t^{(i)}_2}-\frac{\partial\tau}{\partial t^{(p)}_2},\quad 1\leq i\leq p-1.\label{interesting 5}
\end{align}
From these equations, we deduce two families of identities. Let $f:=\log\tau_{\vec{m},\vec{n}}^E(t,s;\alpha,\beta;a,b)$. Firstly, using equation (\ref{interesting 2}) we have
\begin{align}
\sum_{l=1}^{q}a_l\frac{\partial f}{\partial s^{(l)}_1}=-\sum_{l=1}^{q-1}a_l\frac{\partial f}{\partial a_l},\label{sum as}
\end{align}
since $\sum_{l=1}^{q}a_l=0$, and similarly
\begin{gather}
\sum_{l=1}^{q}\alpha_l\frac{\partial f}{\partial s^{(l)}_1}=-\sum_{l=1}^{q-1}\alpha_l\frac{\partial f}{\partial a_l},\qquad\sum_{l=1}^{p}b_l\frac{\partial f}{\partial t^{(l)}_1}=\sum_{l=1}^{p-1}b_l\frac{\partial f}{\partial b_l},\qquad\sum_{l=1}^{p}\beta_l\frac{\partial f}{\partial t^{(l)}_1}=\sum_{l=1}^{p-1}\beta_l\frac{\partial f}{\partial b_l},\notag\\
\sum_{l=1}^{q}\alpha_l\frac{\partial f}{\partial s^{(l)}_2}=-\sum_{l=1}^{q-1}\alpha_l\frac{\partial f}{\partial \alpha_l},\qquad\sum_{l=1}^{p}\beta_l\frac{\partial f}{\partial t^{(l)}_2}=\sum_{l=1}^{p-1}\beta_l\frac{\partial f}{\partial \beta_l}.\label{sum bt}
\end{gather}
Secondly, using equation (\ref{interesting 2}) we have
\begin{align*}
\sum_{i=1}^{q}\frac{\partial f}{\partial s_1^{(i)}}=\sum_{i=1}^{q-1}\frac{\partial f}{\partial a_l}+q\frac{\partial f}{\partial s_1^{(q)}},
\end{align*}
and thus
\begin{align}
\frac{\partial f}{\partial s_1^{(q)}}=\frac{1}{q}\sum_{i=1}^{q}\frac{\partial f}{\partial s_1^{(i)}}-\frac{1}{q}\sum_{i=1}^{q-1}\frac{\partial f}{\partial a_l}.\label{dsq}
\end{align}
Using again equation (\ref{interesting 2}), we obtain
\begin{align}
\frac{\partial f}{\partial s_1^{(j)}}=\frac{1}{q}\sum_{i=1}^{q}\frac{\partial f}{\partial s_1^{(i)}}-\frac{1}{q}\sum_{i=1}^{q-1}\frac{\partial f}{\partial a_l}-\frac{\partial f}{\partial a_j},\quad 1\leq j\leq q-1.\label{dsj}
\end{align}
Equations \eqref{dsq} and \eqref{dsj} can be summarized as follows
\begin{align}
\frac{\partial f}{\partial s_1^{(j)}}=\frac{1}{q}\sum_{i=1}^{q}\frac{\partial f}{\partial s_1^{(i)}}-\frac{1}{q}\sum_{i=1}^{q-1}\frac{\partial f}{\partial a_l}-(1-\delta_{jq})\frac{\partial f}{\partial a_j},\quad 1\leq j\leq q.\label{ds}
\end{align}
Similarly, we have
\begin{align}
&\frac{\partial f}{\partial t_1^{(j)}}=\frac{1}{p}\sum_{i=1}^{p}\frac{\partial f}{\partial t_1^{(i)}}-\frac{1}{p}\sum_{i=1}^{p-1}\frac{\partial f}{\partial b_l}-(1-\delta_{jp})\frac{\partial f}{\partial b_j},\quad 1\leq j\leq p,\notag\\
&\frac{\partial f}{\partial s_2^{(j)}}=\frac{1}{q}\sum_{i=1}^{q}\frac{\partial f}{\partial s_2^{(i)}}-\frac{1}{q}\sum_{i=1}^{q-1}\frac{\partial f}{\partial \alpha_l}-(1-\delta_{jq})\frac{\partial f}{\partial \alpha_j},\quad 1\leq j\leq q,\label{dt}\\
&\frac{\partial f}{\partial t_2^{(j)}}=\frac{1}{p}\sum_{i=1}^{p}\frac{\partial f}{\partial t_2^{(i)}}-\frac{1}{p}\sum_{i=1}^{p-1}\frac{\partial f}{\partial \beta_l}-(1-\delta_{jp})\frac{\partial f}{\partial \beta_j},\quad 1\leq j\leq p.\notag
\end{align}
Substituting relations \eqref{sum as},\eqref{sum bt}, \eqref{ds}, \eqref{dt} in the Virasoro constraints (\ref{Virasoro -1 and 0}), we get
\begin{align}
&A_jf=\frac{\partial f}{\partial s^{(j)}_1}+\frac{1}{q}\sum_{i\geq 2}\Bigg(\sum_{l=1}^{q}is^{(l)}_i\frac{\partial}{\partial s^{(l)}_{i-1}}+\sum_{l=1}^{p}it^{(l)}_i\frac{\partial}{\partial t^{(l)}_{i-1}}\Bigg)f+\frac{1}{q}\Big(\sum_{l=1}^{p}n_lt^{(l)}_1-\sum_{l=1}^{q}m_ls^{(l)}_1\Big)\notag\\
&\phantom{A_jf=}+\frac{1}{q}\Big(\sum_{l=1}^{q}a_lm_l+\sum_{l=1}^{p}b_ln_l\Big),\qquad 1\leq j\leq q,\notag\\
&B_jf=-\frac{\partial f}{\partial t^{(j)}_1}+\frac{1}{p}\sum_{i\geq 2}\Bigg(\sum_{l=1}^{q}is^{(l)}_i\frac{\partial}{\partial s^{(l)}_{i-1}}+\sum_{l=1}^{p}it^{(l)}_i\frac{\partial}{\partial t^{(l)}_{i-1}}\Bigg)f+\frac{1}{p}\Big(\sum_{l=1}^{p}n_lt^{(l)}_1-\sum_{l=1}^{q}m_ls^{(l)}_1\Big)\notag\\
&\phantom{B_jf=}+\frac{1}{p}\Big(\sum_{l=1}^{q}a_lm_l+\sum_{l=1}^{p}b_ln_l\Big),\qquad 1\leq j\leq p,\notag\\
&\hat{A}_jf=\frac{\partial f}{\partial s^{(j)}_2}+\frac{1}{q}\sum_{i\geq 1}\Bigg(\sum_{l=1}^{q}is^{(l)}_i\frac{\partial}{\partial s^{(l)}_{i}}+\sum_{l=1}^{p}it^{(l)}_i\frac{\partial}{\partial t^{(l)}_{i}}\Bigg)f+\frac{K}{q},\qquad 1\leq j\leq q,\notag\\
&\hat{B}_jf=-\frac{\partial f}{\partial t^{(j)}_2}+\frac{1}{p}\sum_{i\geq 1}\Bigg(\sum_{l=1}^{q}is^{(l)}_i\frac{\partial}{\partial s^{(l)}_{i}}+\sum_{l=1}^{p}it^{(l)}_i\frac{\partial}{\partial t^{(l)}_{i}}\Bigg)f+\frac{K}{p},\qquad 1\leq j\leq p,\label{Virasoro bis}
\end{align}
where
\begin{align}
&A_j=-(1-\delta_{jq})\frac{\partial}{\partial a_j}+\frac{1}{q}\Bigg(\mathcal{B}_{-1}+\sum_{l=1}^{q-1}\frac{\partial}{\partial a_l}-2\Big(\sum_{l=1}^{q-1}\alpha_l\frac{\partial}{\partial a_l}+\sum_{l=1}^{p-1}\beta_l\frac{\partial}{\partial b_l}\Big)\Bigg),\qquad 1\leq j\leq q,\notag\\
&B_j=-(1-\delta_{jp})\frac{\partial}{\partial b_j}+\frac{1}{p}\Bigg(\mathcal{B}_{-1}+\sum_{l=1}^{p-1}\frac{\partial}{\partial b_l}-2\Big(\sum_{l=1}^{q-1}\alpha_l\frac{\partial}{\partial a_l}+\sum_{l=1}^{p-1}\beta_l\frac{\partial}{\partial b_l}\Big)\Bigg),\qquad 1\leq j\leq p,\notag
\end{align}
and
\begin{align}
&\hat{A}_j=-(1-\delta_{jq})\frac{\partial}{\partial \alpha_j}+\frac{1}{q}\Bigg(\mathcal{B}_{0}-\Big(\sum_{l=1}^{q-1}a_l\frac{\partial}{\partial a_l}+\sum_{l=1}^{p-1}b_l\frac{\partial}{\partial b_l}\Big)+\sum_{l=1}^{q-1}\frac{\partial}{\partial \alpha_l}-2\Big(\sum_{l=1}^{q-1}\alpha_l\frac{\partial}{\partial \alpha_l}+\sum_{l=1}^{p-1}\beta_l\frac{\partial}{\partial \beta_l}\Big)\Bigg),\notag\\
&\phantom{\hat{A}_j=-(1-\delta_{jq})\frac{\partial}{\partial \alpha_j}+\frac{1}{q}\Bigg(\mathcal{B}_{0}-\Big(\sum_{l=1}^{q-1}a_l\frac{\partial}{\partial a_l}+\sum_{l=1}^{p-1}b_l\frac{\partial}{\partial b_l}\Big)+\sum_{l=1}^{q-1}\frac{\partial}{\partial \alpha_l}-2} 1\leq j\leq q,\notag\\
&\hat{B}_j=-(1-\delta_{jp})\frac{\partial}{\partial \beta_j}+\frac{1}{p}\Bigg(\mathcal{B}_{0}-\Big(\sum_{l=1}^{q-1}a_l\frac{\partial}{\partial a_l}+\sum_{l=1}^{p-1}b_l\frac{\partial}{\partial b_l}\Big)+\sum_{l=1}^{p-1}\frac{\partial}{\partial \beta_l}-2\Big(\sum_{l=1}^{q-1}\alpha_l\frac{\partial}{\partial \alpha_l}+\sum_{l=1}^{p-1}\beta_l\frac{\partial}{\partial \beta_l}\Big)\Bigg),\notag\\
&\phantom{\hat{B}_j=-(1-\delta_{jp})\frac{\partial}{\partial \beta_j}+\frac{1}{p}\Bigg(\mathcal{B}_{0}-\Big(\sum_{l=1}^{q-1}a_l\frac{\partial}{\partial a_l}+\sum_{l=1}^{p-1}b_l\frac{\partial}{\partial b_l}\Big)+\sum_{l=1}^{p-1}\frac{\partial}{\partial \beta_l}-2} 1\leq j\leq p.\notag
\end{align}
Observe that the operators $A_j$, $1\leq j\leq q$, and $B_j$, $1\leq j\leq p$, all commute. 

\begin{lemma}\label{second derivatives}
On the locus $\mathcal{K}=\{(t,s)=0\}$, the function $f:=\log\tau_{\vec{m},\vec{n}}^E(t,s;\alpha,\beta;a,b)$ satisfies the Virasoro constraints 
\begin{align}
&\frac{\partial^2f}{\partial s^{(j)}_1\partial s^{(k)}_1}=A_jA_kf+\frac{m_j}{q}+\frac{m_k}{q}-\frac{N}{q^2}+\frac{2}{q^2}\big(\left\langle\alpha,m \right\rangle+\left\langle \beta,n\right\rangle\big),\notag\\
&\frac{\partial^2f}{\partial t^{(j)}_1\partial t^{(k)}_1}=B_jB_kf+\frac{n_j}{p}+\frac{n_k}{p}-\frac{N}{p^2}+\frac{2}{p^2}\big(\left\langle\alpha,m \right\rangle+\left\langle \beta,n\right\rangle\big),\notag\\
&\frac{\partial^2f}{\partial s^{(j)}_1\partial t^{(k)}_1}=-A_jB_kf-\frac{m_j}{p}-\frac{n_k}{q}+\frac{N}{pq}-\frac{2}{pq}\big(\left\langle \alpha,m\right\rangle+\left\langle \beta,n\right\rangle\big),\notag\\
&\frac{\partial^2f}{\partial s^{(k)}_1\partial s^{(j)}_2}=\Big(\hat{A}_j-\frac{1}{q}\Big)A_kf+\frac{2}{q^2}\big(\left\langle a,m\right\rangle+\left\langle b,n\right\rangle\big),\notag\\
&\frac{\partial^2f}{\partial t^{(k)}_1\partial t^{(j)}_2}=\Big(\hat{B}_j-\frac{1}{p}\Big)B_kf+\frac{2}{p^2}\big(\left\langle a,m\right\rangle+\left\langle b,n\right\rangle\big),\notag\\
&\frac{\partial^2f}{\partial s^{(k)}_1\partial t^{(j)}_2}=-\Big(\hat{B}_j-\frac{1}{p}\Big)A_kf-\frac{2}{pq}\big(\left\langle a,m\right\rangle+\left\langle b,n\right\rangle\big),\notag\\
&\frac{\partial^2f}{\partial s^{(j)}_2\partial t^{(k)}_1}=-\Big(\hat{A}_j-\frac{1}{q}\Big)B_kf-\frac{2}{pq}\big(\left\langle a,m\right\rangle+\left\langle b,n\right\rangle\big),\notag
\end{align}
where $\left\langle\alpha,m \right\rangle=\sum_{i=1}^{q}\alpha_i m_i$ and $\left\langle \beta,n\right\rangle=\sum_{i=1}^{p}\beta_i n_i$.
\end{lemma}

\begin{proof}
We compute on the locus $\mathcal{K}$, using (\ref{Virasoro bis}) that
\begin{align}
&A_jA_kf\big|_{\mathcal{K}}=A_j\Bigg[\frac{\partial f}{\partial s^{(k)}_1}+\frac{1}{q}\sum_{i\geq 2}\Bigg(\sum_{l=1}^{q}is^{(l)}_i\frac{\partial}{\partial s^{(l)}_{i-1}}+\sum_{l=1}^{p}it^{(l)}_i\frac{\partial}{\partial t^{(l)}_{i-1}}\Bigg)f+\frac{1}{q}\Big(\sum_{l=1}^{p}n_lt^{(l)}_1-\sum_{l=1}^{q}m_ls^{(l)}_1\Big)\notag\\
&\phantom{A_jA_kf\big|_{\mathcal{K}}=A_j\Bigg[}+\frac{1}{q}\Big(\sum_{l=1}^{q}a_lm_l+\sum_{l=1}^{p}b_ln_l\Big)\Bigg]\Bigg|_{\mathcal{K}}\notag\\
&\phantom{A_jA_kf\big|_{\mathcal{K}}}=\Bigg[\frac{\partial}{\partial s^{(k)}_1}+\frac{1}{q}\sum_{i\geq 2}\Bigg(\sum_{l=1}^{q}is^{(l)}_i\frac{\partial}{\partial s^{(l)}_{i-1}}+\sum_{l=1}^{p}it^{(l)}_i\frac{\partial}{\partial t^{(l)}_{i-1}}\Bigg)\Bigg]A_jf\big|_{\mathcal{K}}+\frac{1}{q}A_j\Big(\sum_{l=1}^{q}a_lm_l+\sum_{l=1}^{p}b_ln_l\Big)\Big|_{\mathcal{K}}\notag\\
&\phantom{A_jA_kf\big|_{\mathcal{K}}}=\frac{\partial}{\partial s_1^{(k)}}\Bigg[\frac{\partial f}{\partial s^{(j)}_1}+\frac{1}{q}\sum_{i\geq 2}\Bigg(\sum_{l=1}^{q}is^{(l)}_i\frac{\partial}{\partial s^{(l)}_{i-1}}+\sum_{l=1}^{p}it^{(l)}_i\frac{\partial}{\partial t^{(l)}_{i-1}}\Bigg)f+\frac{1}{q}\Big(\sum_{l=1}^{p}n_lt^{(l)}_1-\sum_{l=1}^{q}m_ls^{(l)}_1\Big)\notag\\
&\phantom{A_jA_kf\big|_{\mathcal{K}}=\frac{\partial}{\partial s_1^{(k)}}\Bigg[}+\frac{1}{q}\Big(\sum_{l=1}^{q}a_lm_l+\sum_{l=1}^{p}b_ln_l\Big)\Bigg]\Bigg|_{\mathcal{K}}\notag\\
&\phantom{A_jA_kf\big|_{\mathcal{K}}=}+\frac{1}{q}\Bigg[-(1-\delta_{j,q})\frac{\partial}{\partial a_j}+\frac{1}{q}\Bigg(\mathcal{B}_{-1}+\sum_{l=1}^{q-1}\frac{\partial}{\partial a_l}-2\Big(\sum_{l=1}^{q-1}\alpha_l\frac{\partial}{\partial a_l}+\sum_{l=1}^{p-1}\beta_l\frac{\partial}{\partial b_l}\Big)\Bigg)\Bigg]\notag\\
&\phantom{A_jA_kf\big|_{\mathcal{K}}=\frac{\partial}{\partial s_1^{(k)}}\Bigg[}\Big(\sum_{l=1}^{q}a_lm_l+\sum_{l=1}^{p}b_ln_l\Big)\Big|_{\mathcal{K}}\notag\\
&\phantom{A_jA_kf\big|_{\mathcal{K}}}=\frac{\partial^2f}{\partial s^{(j)}_1\partial s^{(k)}_1}-\frac{m_k}{q}+\frac{1}{q^2}\Big(-q(m_j-m_q)(1-\delta_{j,q})+\sum_{l=1}^{q-1}(m_l-m_q)\notag\\
&\phantom{A_jA_kf\big|_{\mathcal{K}}=}-2\sum_{l=1}^{q-1}\alpha_l(m_l-m_q)-2\sum_{l=1}^{p-1}\beta_l(n_l-n_p)\Big).\notag
\end{align}
Since $\sum_{l=1}^{q-1}(m_l-m_q)=N-qm_q$, $\sum_{l=1}^{q-1}\alpha_l(m_l-m_q)=\left\langle \alpha,m\right\rangle$ and $\sum_{l=1}^{p-1}\beta_l(n_l-n_p)=\left\langle \beta,n\right\rangle$, we obtain
\begin{align}
&A_jA_kf\big|_{\mathcal{K}}=\frac{\partial^2f}{\partial s^{(j)}_1\partial s^{(k)}_1}-\frac{m_j}{q}-\frac{m_k}{q}+\frac{N}{q^2}-\frac{2}{q^2}\big(\left\langle\alpha,m \right\rangle+\left\langle \beta,n\right\rangle\big).\notag
\end{align}
The proof of the other relations is analogous.
\end{proof}


\section{Existence of a PDE for $\log P_{p,q}(E;a,b)$}

In this section we prove that, under the assumptions $a_1+\dots,+a_q=0$ and $b_1+\dots+b_p=0$, the function $\log P_{p,q}(E;a,b)$, with $P_{p,q}(E;a,b)$ as defined in (\ref{normalized problem}), satisfies a nonlinear PDE, the variables being $a_1,\dots, a_{q-1}$, $b_1,\dots,b_{p-1}$ and the coordinates of the endpoints of the set $E$, i.e. $c_1,\dots,c_{2r}$. To perform this, we first show that the function $f:=\log\tau_{\vec{m},\vec{n}}^E(t,s;\alpha,\beta;a,b)$ satisfies a system of $\frac{1}{2}(p+q)(p+q-1)$ equations on the locus $\mathcal{L}$, containing partial derivatives with respect to $a_1,\dots, a_{q-1}$, $b_1,\dots,b_{p-1}$, $\alpha_1,\dots, \alpha_{q-1}$, $\beta_1,\dots,\beta_{p-1}$ and the coordinates of the endpoints of the set $E$.\\

We define the operators
\begin{align}
&A_{j}^{\mathcal{L}}=-(1-\delta_{jq})\frac{\partial}{\partial a_j}+\frac{1}{q}\Big(\mathcal{B}_{-1}+\sum_{l=1}^{q-1}\frac{\partial}{\partial a_l}\Big),\quad 1\leq j\leq q,\notag\\
&B_{j}^{\mathcal{L}}=-(1-\delta_{jp})\frac{\partial}{\partial b_j}+\frac{1}{p}\Big(\mathcal{B}_{-1}+\sum_{l=1}^{p-1}\frac{\partial}{\partial b_l}\Big),\quad 1\leq j\leq p,\label{A^L and B^L}\\
&\hat{\mathcal{B}}_0=\mathcal{B}_{0}-\Big(\sum_{l=1}^{q-1}a_l\frac{\partial}{\partial a_l}+\sum_{l=1}^{p-1}b_l\frac{\partial}{\partial b_l}\Big).\notag
\end{align}
We then have
\begin{align}
&A_j=A_{j}^{\mathcal{L}}-\frac{2}{q}\Big(\sum_{l=1}^{q-1}\alpha_l\frac{\partial}{\partial a_l}+\sum_{l=1}^{p-1}\beta_l\frac{\partial}{\partial b_l}\Big),\quad 1\leq j\leq q,\notag\\
&B_j=B_{j}^{\mathcal{L}}-\frac{2}{p}\Big(\sum_{l=1}^{q-1}\alpha_l\frac{\partial}{\partial a_l}+\sum_{l=1}^{p-1}\beta_l\frac{\partial}{\partial b_l}\Big),\quad 1\leq j\leq p,\notag\\
&\hat{A}_{j}=\frac{1}{q}\hat{\mathcal{B}}_0-(1-\delta_{jq})\frac{\partial}{\partial \alpha_j}-\frac{2}{q}\Big(\sum_{l=1}^{q-1}\alpha_l\frac{\partial}{\partial \alpha_l}+\sum_{l=1}^{p-1}\beta_l\frac{\partial}{\partial \beta_l}\Big)+\frac{1}{q}\sum_{l=1}^{q-1}\frac{\partial}{\partial \alpha_l},\quad 1\leq j\leq q,\label{A^L and B^L bis}\\
&\hat{B}_j=\frac{1}{p}\hat{\mathcal{B}}_0-(1-\delta_{jp})\frac{\partial}{\partial \beta_j}-\frac{2}{p}\Big(\sum_{l=1}^{q-1}\alpha_l\frac{\partial}{\partial \alpha_l}+\sum_{l=1}^{p-1}\beta_l\frac{\partial}{\partial \beta_l}\Big)+\frac{1}{p}\sum_{l=1}^{p-1}\frac{\partial}{\partial \beta_l},\quad 1\leq j\leq p.\notag
\end{align}
We also introduce the following notation
\begin{align}
&-\partial_{\beta_j}+\frac{1}{p}\partial_\beta=-(1-\delta_{jp})\frac{\partial}{\partial \beta_j}+\frac{1}{p}\sum_{l=1}^{p-1}\frac{\partial}{\partial \beta_l},\qquad -\partial_{b_j}+\frac{1}{p}\partial_b=-(1-\delta_{jp})\frac{\partial}{\partial b_j}+\frac{1}{p}\sum_{l=1}^{p-1}\frac{\partial}{\partial b_l},\notag\\
&-\partial_{\alpha_k}+\frac{1}{q}\partial_\alpha=-(1-\delta_{kq})\frac{\partial}{\partial \alpha_k}+\frac{1}{q}\sum_{l=1}^{q-1}\frac{\partial}{\partial \alpha_l},\qquad -\partial_{a_k}+\frac{1}{q}\partial_a=-(1-\delta_{kq})\frac{\partial}{\partial a_k}+\frac{1}{q}\sum_{l=1}^{q-1}\frac{\partial}{\partial a_l},\label{notation derivatives}
\end{align}
with $1\leq j\leq p$ and $1\leq k\leq q$. Note the two sets of operators on the first row respectively sum to zero as we sum $j$ from $1$ to $p$, and similarly for the operators on the second row, as we sum $k$ from $1$ to $q$. With these notations we have the following.

\begin{theorem}\label{system of equations for log tau with unknowns}
The function $f:=\log\tau_{\vec{m},\vec{n}}^E(t,s;\alpha,\beta;a,b)$ satisfies the following $\frac{1}{2}(p+q)(p+q-1)$ equations\footnote{$\{f,g\}_X=g\,X(f)-f\,X(g)$} on the locus $\mathcal{L}$
\begin{multline*}
\Big\{A_{j}^{\mathcal{L}}(\frac{1}{p}\partial_\beta-\partial_{\beta_k})f,A_{j}^{\mathcal{L}}B_{k}^{\mathcal{L}}f+\frac{m_j}{p}+\frac{n_k}{q}-\frac{N}{pq}\Big\}_{A_{j}^{\mathcal{L}}}-\Big\{B_{k}^{\mathcal{L}}(\frac{1}{q}\partial_\alpha-\partial_{\alpha_j})f,A_{j}^{\mathcal{L}}B_{k}^{\mathcal{L}}f+\frac{m_j}{p}+\frac{n_k}{q}-\frac{N}{pq}\Big\}_{B_{k}^{\mathcal{L}}}\\
=G_{jk}^{AB},\quad 1\leq j\leq q, 1\leq k\leq p,
\end{multline*}
\begin{multline}\label{summary}
\Big\{A_{k}^{\mathcal{L}}(\frac{1}{q}\partial_\alpha-\partial_{\alpha_j})f,A_{j}^{\mathcal{L}}A_{k}^{\mathcal{L}}f+\frac{m_j}{q}+\frac{m_k}{q}-\frac{N}{q^2}\Big\}_{A_{k}^{\mathcal{L}}}+\Big\{A_{j}^{\mathcal{L}}(\frac{1}{q}\partial_\alpha-\partial_{\alpha_k})f,A_{j}^{\mathcal{L}}A_{k}^{\mathcal{L}}f+\frac{m_j}{q}+\frac{m_k}{q}-\frac{N}{q^2}\Big\}_{A_{j}^{\mathcal{L}}}\\
=G_{jk}^{A},\quad 1\leq j<k\leq q,
\end{multline}
\begin{multline*}
\Big\{B_{k}^{\mathcal{L}}(\frac{1}{p}\partial_\beta-\partial_{\beta_j})f,B_{j}^{\mathcal{L}}B_{k}^{\mathcal{L}}f+\frac{n_j}{p}+\frac{n_k}{p}-\frac{N}{p^2}\Big\}_{B_{k}^{\mathcal{L}}}+\Big\{B_{j}^{\mathcal{L}}(\frac{1}{p}\partial_\beta-\partial_{\beta_k})f,B_{j}^{\mathcal{L}}B_{k}^{\mathcal{L}}f+\frac{n_j}{p}+\frac{n_k}{p}-\frac{N}{p^2}\Big\}_{B_{j}^{\mathcal{L}}}\\
=G_{jk}^{B},\quad 1\leq j<k\leq p,\notag
\end{multline*}
where $G_{jk}^{A}$, $G_{jk}^{B}$ and $G_{jk}^{AB}$ only depend on $f$, its derivatives with respect to $a_1,\dots, a_{q-1}$, $b_1,\dots,b_{p-1}$, and its differentials up to the third order with respect to the operators $A_{j}^{\mathcal{L}}$, $B_{j}^{\mathcal{L}}$ and $\hat{\mathcal{B}}_0$, evaluated on the locus $\mathcal{L}$.
\end{theorem}

\begin{proof}
Using equations (\ref{Virasoro bis}) we obtain on the locus $\mathcal{K}$
\begin{align}
&\frac{\partial}{\partial s^{(j)}_1}\log\frac{\tau_{\vec{m}-\vec{e_j}+\vec{e_k},\vec{n}}^E}{\tau_{\vec{m}+\vec{e_j}-\vec{e_k},\vec{n}}^E}\Bigg|_{\mathcal{K}}=A_j\log\frac{\tau_{\vec{m}-\vec{e_j}+\vec{e_k},\vec{n}}^E}{\tau_{\vec{m}+\vec{e_j}-\vec{e_k},\vec{n}}^E}+\frac{2}{q}(a_j-a_k),\notag\\
&\frac{\partial}{\partial s^{(j)}_1}\log\frac{\tau_{\vec{m}-\vec{e_j},\vec{n}-\vec{e_k}}^E}{\tau_{\vec{m}+\vec{e_j},\vec{n}+\vec{e_k}}^E}\Bigg|_{\mathcal{K}}=A_j\log\frac{\tau_{\vec{m}-\vec{e_j},\vec{n}-\vec{e_k}}^E}{\tau_{\vec{m}+\vec{e_j},\vec{n}+\vec{e_k}}^E}+\frac{2}{q}(a_j+b_k),\notag\\
&\frac{\partial}{\partial t^{(j)}_1}\log\frac{\tau_{\vec{m},\vec{n}+\vec{e_j}-\vec{e_k}}^E}{\tau_{\vec{m},\vec{n}-\vec{e_j}+\vec{e_k}}^E}\Bigg|_{\mathcal{K}}=-B_j\log\frac{\tau_{\vec{m},\vec{n}+\vec{e_j}-\vec{e_k}}^E}{\tau_{\vec{m},\vec{n}-\vec{e_j}+\vec{e_k}}^E}+\frac{2}{p}(b_j-b_k),\notag\\
&\frac{\partial}{\partial t^{(j)}_1}\log\frac{\tau_{\vec{m}+\vec{e_k},\vec{n}+\vec{e_j}}^E}{\tau_{\vec{m}-\vec{e_k},\vec{n}-\vec{e_j}}^E}\Bigg|_{\mathcal{K}}=-B_j\log\frac{\tau_{\vec{m}+\vec{e_k},\vec{n}+\vec{e_j}}^E}{\tau_{\vec{m}-\vec{e_k},\vec{n}-\vec{e_j}}^E}+\frac{2}{p}(a_k+b_j).\label{A_j difference log 1}
\end{align}

Substituting the first equation in (\ref{A_j difference log 1}) into the second equation in (\ref{bilinear}) and using lemma \ref{second derivatives}, we have on the locus $\mathcal{K}$
\begin{align}
A_j\log\frac{\tau_{\vec{m}-\vec{e_j}+\vec{e_k},\vec{n}}^E}{\tau_{\vec{m}+\vec{e_j}-\vec{e_k},\vec{n}}^E}=\frac{\Big(\hat{A}_j-\frac{1}{q}\Big)A_kf+\frac{2}{q^2}\big(\left\langle a,m\right\rangle+\left\langle b,n\right\rangle\big)}{A_jA_kf+\frac{m_j}{q}+\frac{m_k}{q}-\frac{N}{q^2}+\frac{2}{q^2}\big(\left\langle\alpha,m \right\rangle+\left\langle \beta,n\right\rangle\big)}-\frac{2}{q}(a_j-a_k).\label{bilinear 1 bis}
\end{align}
Similarly, we have
\begin{align}
&B_j\,\ln\frac{\tau_{\vec{m},\vec{n}-\vec{e_j}+\vec{e_{k}}}^E}{\tau_{\vec{m},\vec{n}+\vec{e_j}-\vec{e_{k}}}^E}=\frac{\Big(\hat{B}_j-\frac{1}{p}\Big)B_kf+\frac{2}{p^2}\big(\left\langle a,m\right\rangle+\left\langle b,n\right\rangle\big)}{B_jB_kf+\frac{n_j}{p}+\frac{n_k}{p}-\frac{N}{p^2}+\frac{2}{p^2}\big(\left\langle\alpha,m \right\rangle+\left\langle \beta,n\right\rangle\big)}-\frac{2}{p}(b_j-b_k),\label{bilinear 2 bis}\\
&B_j\,\ln\frac{\tau_{\vec{m}-\vec{e_k},\vec{n}-\vec{e_j}}^E}{\tau_{\vec{m}+\vec{e_k},\vec{n}+\vec{e_j}}^E}=\frac{-\Big(\hat{B}_j-\frac{1}{p}\Big)A_kf-\frac{2}{pq}\big(\left\langle a,m\right\rangle+\left\langle b,n\right\rangle\big)}{-A_kB_jf-\frac{m_k}{p}-\frac{n_j}{q}+\frac{N}{pq}-\frac{2}{pq}\big(\left\langle \alpha,m\right\rangle+\left\langle \beta,n\right\rangle\big)}-\frac{2}{p}(a_k+b_j),\label{bilinear 3 bis}\\
&A_j\,\ln\frac{\tau_{\vec{m}-\vec{e_j},\vec{n}-\vec{e_k}}^E}{\tau_{\vec{m}+\vec{e_j},\vec{n}+\vec{e_k}}^E}=\frac{-\Big(\hat{A}_j-\frac{1}{q}\Big)B_kf-\frac{2}{pq}\big(\left\langle a,m\right\rangle+\left\langle b,n\right\rangle\big)}{-A_jB_kf-\frac{m_j}{p}-\frac{n_k}{q}+\frac{N}{pq}-\frac{2}{pq}\big(\left\langle \alpha,m\right\rangle+\left\langle \beta,n\right\rangle\big)}-\frac{2}{q}(a_j+b_k).\label{bilinear 4 bis}
\end{align}
Let us denote equations (\ref{bilinear 1 bis})-(\ref{bilinear 4 bis}), with indices chosen as above, by $(\ref{bilinear 1 bis})_{jk}$, $(\ref{bilinear 2 bis})_{jk}$, $(\ref{bilinear 3 bis})_{jk}$ and $(\ref{bilinear 4 bis})_{jk}$. We compute $A_j(\ref{bilinear 3 bis})_{kj}-B_k(\ref{bilinear 4 bis})_{jk}$, and we obtain
\begin{multline*}
0=A_j\Bigg(\frac{-\Big(\hat{B}_k-\frac{1}{p}\Big)A_jf-\frac{2}{pq}\big(\left\langle a,m\right\rangle+\left\langle b,n\right\rangle\big)}{-A_jB_kf-\frac{m_j}{p}-\frac{n_k}{q}+\frac{N}{pq}-\frac{2}{pq}\big(\left\langle \alpha,m\right\rangle+\left\langle \beta,n\right\rangle\big)}-\frac{2}{p}(a_j+b_k)\Bigg)\\
-B_k\Bigg(\frac{-\Big(\hat{A}_j-\frac{1}{q}\Big)B_kf-\frac{2}{pq}\big(\left\langle a,m\right\rangle+\left\langle b,n\right\rangle\big)}{-A_jB_kf-\frac{m_j}{p}-\frac{n_k}{q}+\frac{N}{pq}-\frac{2}{pq}\big(\left\langle \alpha,m\right\rangle+\left\langle \beta,n\right\rangle\big)}-\frac{2}{q}(a_j+b_k)\Bigg),\notag
\end{multline*} 
since $[A_j,B_k]=0$. Define $\sigma(a,b)=\left\langle a,m\right\rangle+\left\langle b,n\right\rangle$. As $A_j$ and $B_k$ are first order differential operators, we have
\begin{multline}\label{bilinear tris 1}
0=\frac{\Big\{\Big(\hat{B}_k-\frac{1}{p}\Big)A_jf+\frac{2\sigma(a,b)}{pq},A_jB_kf+\frac{m_j}{p}+\frac{n_k}{q}-\frac{N}{pq}+\frac{2\sigma(\alpha,\beta)}{pq}\Big\}_{A_j}}{\Big(-A_jB_kf-\frac{m_j}{p}-\frac{n_k}{q}+\frac{N}{pq}-\frac{2\sigma(\alpha,\beta)}{pq}\Big)^2}\\
-\frac{\Big\{\Big(\hat{A}_j-\frac{1}{q}\Big)B_kf+\frac{2\sigma(a,b)}{pq},A_jB_kf+\frac{m_j}{p}+\frac{n_k}{q}-\frac{N}{pq}+\frac{2\sigma(\alpha,\beta)}{pq}\Big\}_{B_k}}{\Big(-A_jB_kf-\frac{m_j}{p}-\frac{n_k}{q}+\frac{N}{pq}-\frac{2\sigma(\alpha,\beta)}{pq}\Big)^2}-\frac{2}{p}A_j(a_j+b_k)+\frac{2}{q}B_k(a_j+b_k).
\end{multline}
Similarly, we compute, for $j\neq k$, $A_k(\ref{bilinear 1 bis})_{jk}+A_j(\ref{bilinear 1 bis})_{kj}$ and $B_k(\ref{bilinear 2 bis})_{jk}+B_j(\ref{bilinear 2 bis})_{kj}$, and we obtain
\begin{multline}\label{bilinear tris 2}
0=\frac{\Big\{\Big(\hat{A}_j-\frac{1}{q}\Big)A_kf+\frac{2\sigma(a,b)}{q^2},A_jA_kf+\frac{m_j}{q}+\frac{m_k}{q}-\frac{N}{q^2}+\frac{2\sigma(\alpha,\beta)}{q^2}\Big\}_{A_k}}{\Big(A_jA_kf+\frac{m_j}{q}+\frac{m_k}{q}-\frac{N}{q^2}+\frac{2\sigma(\alpha,\beta)}{q^2}\Big)^2}\\
+\frac{\Big\{\Big(\hat{A}_k-\frac{1}{q}\Big)A_jf+\frac{2\sigma(a,b)}{q^2},A_kA_jf+\frac{m_k}{q}+\frac{m_j}{q}-\frac{N}{q^2}+\frac{2\sigma(\alpha,\beta)}{q^2}\Big\}_{A_j}}{\Big(A_jA_kf+\frac{m_j}{q}+\frac{m_k}{q}-\frac{N}{q^2}+\frac{2\sigma(\alpha,\beta)}{q^2}\Big)^2}-\frac{4}{q},
\end{multline}
and
\begin{multline}\label{bilinear tris 3}
0=\frac{\Big\{\Big(\hat{B}_j-\frac{1}{p}\Big)B_kf+\frac{2\sigma(a,b)}{p^2},B_jB_kf+\frac{n_j}{p}+\frac{n_k}{p}-\frac{N}{p^2}+\frac{2\sigma(\alpha,\beta)}{p^2}\Big\}_{B_k}}{\Big(B_jB_kf+\frac{n_j}{p}+\frac{n_k}{p}-\frac{N}{p^2}+\frac{2\sigma(\alpha,\beta)}{p^2}\Big)^2}\\
+\frac{\Big\{\Big(\hat{B}_k-\frac{1}{p}\Big)B_jf+\frac{2\sigma(a,b)}{p^2},B_kB_jf+\frac{n_k}{p}+\frac{n_j}{p}-\frac{N}{p^2}+\frac{2\sigma(\alpha,\beta)}{p^2}\Big\}_{B_j}}{\Big(B_kB_jf+\frac{n_k}{p}+\frac{n_j}{p}-\frac{N}{p^2}+\frac{2\sigma(\alpha,\beta)}{p^2}\Big)^2}-\frac{4}{p}.
\end{multline}

Using (\ref{A^L and B^L bis}) we compute
\begin{align}
A_jA_kf&=\Bigg(A_{j}^{\mathcal{L}}-\frac{2}{q}\Big(\sum_{l=1}^{q-1}\alpha_l\frac{\partial}{\partial a_l}+\sum_{l=1}^{p-1}\beta_l\frac{\partial}{\partial b_l}\Big)\Bigg)\Bigg(A_{k}^{\mathcal{L}}-\frac{2}{q}\Big(\sum_{l=1}^{q-1}\alpha_l\frac{\partial}{\partial a_l}+\sum_{l=1}^{p-1}\beta_l\frac{\partial}{\partial b_l}\Big)\Bigg)\Bigg)f\notag\\
&=A_{j}^{\mathcal{L}}A_{k}^{\mathcal{L}}f+\mathcal{O}(\alpha,\beta).\notag
\end{align}
Similarly we have
\begin{align}
&A_jB_kf=A_{j}^{\mathcal{L}}B_{k}^{\mathcal{L}}f+\mathcal{O}(\alpha,\beta),&B_jB_kf=B_{j}^{\mathcal{L}}B_{k}^{\mathcal{L}}f+\mathcal{O}(\alpha,\beta),\notag
\end{align}
and
\begin{align}
&\Big(\hat{A}_j-\frac{1}{q}\Big)A_kf=\Big(\frac{1}{q}\hat{\mathcal{B}}_0-(\partial_{\alpha_j}-\frac{1}{q}\partial_\alpha)-\frac{1}{q}\Big)A_{k}^{\mathcal{L}}f+\frac{2}{q}(\partial_{a_j}-\frac{1}{q}\partial_a)f+\mathcal{O}(\alpha,\beta),\notag\\
&\Big(\hat{A}_j-\frac{1}{q}\Big)B_kf=\Big(\frac{1}{q}\hat{\mathcal{B}}_0-(\partial_{\alpha_j}-\frac{1}{q}\partial_\alpha)-\frac{1}{q}\Big)B_{k}^{\mathcal{L}}f+\frac{2}{p}(\partial_{a_j}-\frac{1}{q}\partial_a)f+\mathcal{O}(\alpha,\beta),\notag\\
&\Big(\hat{B}_j-\frac{1}{p}\Big)A_kf=\Big(\frac{1}{p}\hat{\mathcal{B}}_0-(\partial_{\beta_j}-\frac{1}{p}\partial_\beta)-\frac{1}{p}\Big)A_{k}^{\mathcal{L}}f+\frac{2}{q}(\partial_{b_j}-\frac{1}{p}\partial_b)f+\mathcal{O}(\alpha,\beta),\notag\\
&\Big(\hat{B}_j-\frac{1}{p}\Big)B_kf=\Big(\frac{1}{p}\hat{\mathcal{B}}_0-(\partial_{\beta_j}-\frac{1}{p}\partial_\beta)-\frac{1}{p}\Big)B_{k}^{\mathcal{L}}f+\frac{2}{p}(\partial_{b_j}-\frac{1}{p}\partial_b)f+\mathcal{O}(\alpha,\beta).\notag
\end{align}
Consequently, on the locus $\mathcal{L}$ the equation (\ref{bilinear tris 1}) can be written
\begin{multline*}
0=\frac{\Big\{\Big(\frac{1}{p}\hat{\mathcal{B}}_0-(\partial_{\beta_k}-\frac{1}{p}\partial_\beta)-\frac{1}{p}\Big)A_{j}^{\mathcal{L}}f+\frac{2}{q}(\partial_{b_k}-\frac{1}{p}\partial_b)f+\frac{2\sigma(a,b)}{pq},A_{j}^{\mathcal{L}}B_{k}^{\mathcal{L}}f+\frac{m_j}{p}+\frac{n_k}{q}-\frac{N}{pq}\Big\}_{A_{j}^{\mathcal{L}}}}{\Big(-A_{j}^{\mathcal{L}}B_{k}^{\mathcal{L}}f-\frac{m_j}{p}-\frac{n_k}{q}+\frac{N}{pq}\Big)^2}\\
-\frac{\Big\{\Big(\frac{1}{q}\hat{\mathcal{B}}_0-(\partial_{\alpha_j}-\frac{1}{q}\partial_\alpha)-\frac{1}{q}\Big)B_{k}^{\mathcal{L}}f+\frac{2}{p}(\partial_{a_j}-\frac{1}{q}\partial_a)f+\frac{2\sigma(a,b)}{pq},A_{j}^{\mathcal{L}}B_{k}^{\mathcal{L}}f+\frac{m_j}{p}+\frac{n_k}{q}-\frac{N}{pq}\Big\}_{B_{k}^{\mathcal{L}}}}{\Big(-A_{j}^{\mathcal{L}}B_{k}^{\mathcal{L}}f-\frac{m_j}{p}-\frac{n_k}{q}+\frac{N}{pq}\Big)^2}+\frac{2}{p}-\frac{2}{q}.
\end{multline*}
Putting all the terms which do not contain derivatives of $f$ with respect to $\alpha_i$'s or $\beta_j$'s in the left hand side, we obtain
\begin{multline*}
G_{jk}^{AB}=\Big\{(\frac{1}{p}\partial_\beta-\partial_{\beta_k})A_{j}^{\mathcal{L}}f,A_{j}^{\mathcal{L}}B_{k}^{\mathcal{L}}f+\frac{m_j}{p}+\frac{n_k}{q}-\frac{N}{pq}\Big\}_{A_{j}^{\mathcal{L}}}\\
-\Big\{(\frac{1}{q}\partial_\alpha-\partial_{\alpha_j})B_{k}^{\mathcal{L}}f,A_{j}^{\mathcal{L}}B_{k}^{\mathcal{L}}f+\frac{m_j}{p}+\frac{n_k}{q}-\frac{N}{pq}\Big\}_{B_{k}^{\mathcal{L}}},
\end{multline*}
where
\begin{align}
&G_{jk}^{AB}:=\Big(\frac{2}{q}-\frac{2}{p}\Big)\Big(-A_{j}^{\mathcal{L}}B_{k}^{\mathcal{L}}f-\frac{m_j}{p}-\frac{n_k}{q}+\frac{N}{pq}\Big)^2\notag\\
&\phantom{G_{ij}:=}-\Big\{\Big(\frac{1}{p}\hat{\mathcal{B}}_0-\frac{1}{p}\Big)A_{j}^{\mathcal{L}}f+\frac{2}{q}(\partial_{b_k}-\frac{1}{p}\partial_b)f+\frac{2\sigma(a,b)}{pq},A_{j}^{\mathcal{L}}B_{k}^{\mathcal{L}}f+\frac{m_j}{p}+\frac{n_k}{q}-\frac{N}{pq}\Big\}_{A_{j}^{\mathcal{L}}}\notag\\
&\phantom{G_{ij}:=}+\Big\{\Big(\frac{1}{q}\hat{\mathcal{B}}_0-\frac{1}{q}\Big)B_{k}^{\mathcal{L}}f+\frac{2}{p}(\partial_{a_j}-\frac{1}{q}\partial_a)f+\frac{2\sigma(a,b)}{pq},A_{j}^{\mathcal{L}}B_{k}^{\mathcal{L}}f+\frac{m_j}{p}+\frac{n_k}{q}-\frac{N}{pq}\Big\}_{B_{k}^{\mathcal{L}}}.\notag
\end{align}

Similarly, on the locus $\mathcal{L}$, the equations (\ref{bilinear tris 2}) and (\ref{bilinear tris 3}) can be written
\begin{multline*}
G_{jk}^{A}=\Big\{(\frac{1}{q}\partial_\alpha-\partial_{\alpha_j})A_{k}^{\mathcal{L}}f,A_{j}^{\mathcal{L}}A_{k}^{\mathcal{L}}f+\frac{m_j}{q}+\frac{m_k}{q}-\frac{N}{q^2}\Big\}_{A_{k}^{\mathcal{L}}}\\
+\Big\{(\frac{1}{q}\partial_\alpha-\partial_{\alpha_k})A_{j}^{\mathcal{L}}f,A_{j}^{\mathcal{L}}A_{k}^{\mathcal{L}}f+\frac{m_j}{q}+\frac{m_k}{q}-\frac{N}{q^2}\Big\}_{A_{j}^{\mathcal{L}}},
\end{multline*}
\begin{multline*}
G_{jk}^{B}=\Big\{(\frac{1}{p}\partial_\beta-\partial_{\beta_j})B_{k}^{\mathcal{L}}f,B_{j}^{\mathcal{L}}B_{k}^{\mathcal{L}}f+\frac{n_j}{p}+\frac{n_k}{p}-\frac{N}{p^2}\Big\}_{B_{k}^{\mathcal{L}}}\\
+\Big\{(\frac{1}{p}\partial_\beta-\partial_{\beta_k})B_{j}^{\mathcal{L}}f,B_{j}^{\mathcal{L}}B_{k}^{\mathcal{L}}f+\frac{n_j}{p}+\frac{n_k}{p}-\frac{N}{p^2}\Big\}_{B_{j}^{\mathcal{L}}},
\end{multline*}
where
\begin{align}
&G_{jk}^A:=\frac{4}{q}\Big(A_{j}^{\mathcal{L}}A_{k}^{\mathcal{L}}f+\frac{m_j}{q}+\frac{m_k}{q}-\frac{N}{q^2}\Big)^2\notag\\
&\phantom{G_{jk}^A:=}-\Big\{\Big(\frac{1}{q}\hat{\mathcal{B}}_0-\frac{1}{q}\Big)A_{k}^{\mathcal{L}}f+\frac{2}{q}(\partial_{a_j}-\frac{1}{q}\partial_a)f+\frac{2\sigma(a,b)}{q^2},A_{j}^{\mathcal{L}}A_{k}^{\mathcal{L}}f+\frac{m_j}{q}+\frac{m_k}{q}-\frac{N}{q^2}\Big\}_{A_{k}^{\mathcal{L}}}\notag\\
&\phantom{G_{jk}^A:=}-\Big\{\Big(\frac{1}{q}\hat{\mathcal{B}}_0-\frac{1}{q}\Big)A_{j}^{\mathcal{L}}f+\frac{2}{q}(\partial_{a_k}-\frac{1}{q}\partial_a)f+\frac{2\sigma(a,b)}{q^2},A_{j}^{\mathcal{L}}A_{k}^{\mathcal{L}}f+\frac{m_j}{q}+\frac{m_k}{q}-\frac{N}{q^2}\Big\}_{A_{j}^{\mathcal{L}}},\notag\\
&G_{jk}^B:=\frac{4}{p}\Big(B_{j}^{\mathcal{L}}B_{k}^{\mathcal{L}}f+\frac{n_j}{p}+\frac{n_k}{p}-\frac{N}{p^2}\Big)^2\notag\\
&\phantom{G_{jk}^B:=}-\Big\{\Big(\frac{1}{p}\hat{\mathcal{B}}_0-\frac{1}{p}\Big)B_{k}^{\mathcal{L}}f+\frac{2}{p}(\partial_{b_j}-\frac{1}{p}\partial_b)f+\frac{2\sigma(a,b)}{p^2},B_{j}^{\mathcal{L}}B_{k}^{\mathcal{L}}f+\frac{n_j}{p}+\frac{n_k}{p}-\frac{N}{p^2}\Big\}_{B_{k}^{\mathcal{L}}}\notag\\
&\phantom{G_{jk}^B:=}-\Big\{\Big(\frac{1}{p}\hat{\mathcal{B}}_0-\frac{1}{p}\Big)B_{j}^{\mathcal{L}}f+\frac{2}{p}(\partial_{b_k}-\frac{1}{p}\partial_b)f+\frac{2\sigma(a,b)}{p^2},B_{j}^{\mathcal{L}}B_{k}^{\mathcal{L}}f+\frac{n_j}{p}+\frac{n_k}{p}-\frac{N}{p^2}\Big\}_{B_{j}^{\mathcal{L}}}.\notag
\end{align}
\end{proof}

In order to obtain a PDE for $f=\log\tau_{\vec{m},\vec{n}}^E(0;a,b)$ or for $\log P_{p,q}(E,a,b)$, we need to eliminate the partial derivatives of $f$ with respect to $\alpha_1,\dots, \alpha_{q-1}$, $\beta_1,\dots,\beta_{p-1}$ from the equations (\ref{summary}) in Theorem \ref{system of equations for log tau with unknowns}. Define
\begin{align}
&X_i=(\frac{1}{q}\partial_\alpha-\partial_{\alpha_i})f\big|_{\mathcal{L}},\quad 1\leq i\leq q,\qquad\text{and}\qquad Y_i=(\frac{1}{p}\partial_\beta-\partial_{\beta_i})f\big|_{\mathcal{L}},\quad 1\leq i\leq p.\notag
\end{align}
Note that we have $\sum_{i=1}^{q}X_i=\sum_{i=1}^{p}Y_i=0$, and $\sum_{i=1}^{q}A_i^{\mathcal{L}}=\sum_{i=1}^{p}B_i^{\mathcal{L}}=\mathcal{B}_{-1}$. Consequently, there are among $A_i^{\mathcal{L}}$, $1\leq i\leq q$, and $B_j^{\mathcal{L}}$, $1\leq j\leq p$, only $p+q-1$ linearly independent differential operators. Set $A_q^{\mathcal{L}}=\mathcal{B}_{-1}-\sum_{i=1}^{q-1}A_i^{\mathcal{L}}$ and $B_p^{\mathcal{L}}=\mathcal{B}_{-1}-\sum_{i=1}^{p-1}B_i^{\mathcal{L}}$.\\

With these notations, the equations (\ref{summary}) can be written
\begin{align}
&\Big\{A_{j}^{\mathcal{L}}Y_k,A_{j}^{\mathcal{L}}B_{k}^{\mathcal{L}}f+\frac{m_j}{p}+\frac{n_k}{q}-\frac{N}{pq}\Big\}_{A_{j}^{\mathcal{L}}}-\Big\{B_{k}^{\mathcal{L}}X_j,A_{j}^{\mathcal{L}}B_{k}^{\mathcal{L}}f+\frac{m_j}{p}+\frac{n_k}{q}-\frac{N}{pq}\Big\}_{B_{k}^{\mathcal{L}}}\notag\\
&\phantom{\Big\{A_{j}^{\mathcal{L}}Y_k,A_{j}^{\mathcal{L}}B_{k}^{\mathcal{L}}f+\frac{m_j}{p}+\frac{n_k}{q}-\frac{N}{pq}\Big\}_{A_{j}^{\mathcal{L}}}-\Big\{B_{k}^{\mathcal{L}}X_j}=G_{jk}^{AB},\quad 1\leq j\leq q, 1\leq k\leq p,\notag\\
&\Big\{A_{k}^{\mathcal{L}}X_j,A_{j}^{\mathcal{L}}A_{k}^{\mathcal{L}}f+\frac{m_j}{q}+\frac{m_k}{q}-\frac{N}{q^2}\Big\}_{A_{k}^{\mathcal{L}}}+\Big\{A_{j}^{\mathcal{L}}X_k,A_{j}^{\mathcal{L}}A_{k}^{\mathcal{L}}f+\frac{m_j}{q}+\frac{m_k}{q}-\frac{N}{q^2}\Big\}_{A_{j}^{\mathcal{L}}}\notag\\
&\phantom{\Big\{A_{j}^{\mathcal{L}}Y_k,A_{j}^{\mathcal{L}}B_{k}^{\mathcal{L}}f+\frac{m_j}{p}+\frac{n_k}{q}-\frac{N}{pq}\Big\}_{A_{j}^{\mathcal{L}}}-\Big\{B_{k}^{\mathcal{L}}X_j}=G_{jk}^{A},\quad 1\leq j<k\leq q,\notag\\
&\Big\{B_{k}^{\mathcal{L}}Y_j,B_{j}^{\mathcal{L}}B_{k}^{\mathcal{L}}f+\frac{n_j}{p}+\frac{n_k}{p}-\frac{N}{p^2}\Big\}_{B_{k}^{\mathcal{L}}}+\Big\{B_{j}^{\mathcal{L}}Y_k,B_{j}^{\mathcal{L}}B_{k}^{\mathcal{L}}f+\frac{n_j}{p}+\frac{n_k}{p}-\frac{N}{p^2}\Big\}_{B_{j}^{\mathcal{L}}}\notag\\
&\phantom{\Big\{A_{j}^{\mathcal{L}}Y_k,A_{j}^{\mathcal{L}}B_{k}^{\mathcal{L}}f+\frac{m_j}{p}+\frac{n_k}{q}-\frac{N}{pq}\Big\}_{A_{j}^{\mathcal{L}}}-\Big\{B_{k}^{\mathcal{L}}X_j}=G_{jk}^{B},\quad 1\leq j<k\leq p,\notag
\end{align}
or
\begin{align}
&\big(A_{j}^{\mathcal{L}}\big)^2Y_k-\big(B_{k}^{\mathcal{L}}\big)^2X_j-\Big(\frac{1}{c_{jk}^{AB}}A_{j}^{\mathcal{L}}c_{jk}^{AB}\Big)A_{j}^{\mathcal{L}}Y_k+\Big(\frac{1}{c_{jk}^{AB}}B_{k}^{\mathcal{L}}c_{jk}^{AB}\Big)B_{k}^{\mathcal{L}}X_j=g_{jk}^{AB},\quad 1\leq j\leq q, 1\leq k\leq p,\notag\\
&\big(A_{j}^{\mathcal{L}})^2X_k+\big(A_{k}^{\mathcal{L}}\big)^2X_j-\Big(\frac{1}{c_{jk}^A}A_{j}^{\mathcal{L}}c_{jk}^A\Big)A_{j}^{\mathcal{L}}X_k-\Big(\frac{1}{c_{jk}^A}A_{k}^{\mathcal{L}}c_{jk}^A\Big)A_{k}^{\mathcal{L}}X_j=g_{jk}^{A},\quad 1\leq j<k\leq q,\label{summary 2}\\
&\big(B_{j}^{\mathcal{L}}\big)^2Y_k+\big(B_{k}^{\mathcal{L}}\big)^2Y_j-\Big(\frac{1}{c_{jk}^B}B_{j}^{\mathcal{L}}c_{jk}^B\Big)B_{j}^{\mathcal{L}}Y_k-\Big(\frac{1}{c_{jk}^B}B_{k}^{\mathcal{L}}c_{jk}^B\Big)B_{k}^{\mathcal{L}}Y_j=g_{jk}^{B},\quad 1\leq j<k\leq p,\notag
\end{align}
where
\begin{gather}
c_{jk}^{AB}:=A_{j}^{\mathcal{L}}B_{k}^{\mathcal{L}}f+\frac{m_j}{p}+\frac{n_k}{q}-\frac{N}{pq},\qquad c_{jk}^A:=A_{j}^{\mathcal{L}}A_{k}^{\mathcal{L}}f+\frac{m_j}{q}+\frac{m_k}{q}-\frac{N}{q^2},\notag\\
c_{jk}^B:=B_{j}^{\mathcal{L}}B_{k}^{\mathcal{L}}f+\frac{n_j}{p}+\frac{n_k}{p}-\frac{N}{p^2},\label{c_ij}
\end{gather}
and
\begin{align}
g_{jk}^{AB}:=\frac{G_{jk}^{AB}}{c_{jk}^{AB}},\qquad g_{jk}^{A}:=\frac{G_{jk}^{A}}{c_{jk}^{A}},\qquad g_{jk}^B:=\frac{G_{jk}^{B}}{c_{jk}^{B}}.\label{g_ij}
\end{align}
We have thus a system of $M=\frac{1}{2}(p+q)(p+q-1)$ linear equations in the $p+q-2$ unknown functions $X_1,\dots,X_{q-1}$ and $Y_1,\dots,Y_{p-1}$ and at most all their first and second order derivatives with respect to the independent commuting differential operators $A_{i}^{\mathcal{L}}$, $1\leq i\leq q-1$, $B_{j}^{\mathcal{L}}$, $1\leq j\leq p-1$, and $\mathcal{B}_{-1}$. We think at all these quantities as unknowns. At this point, we have a system with a smaller number of linear equations then unknowns. The general strategy is to keep differentiating the equations and show that at some point we must reach a balance between the number of equations and the number of unknowns, leading to the vanishing of a determinant \emph{at the first point this occurs}, which must yield a nontrivial relation. Let $Z_M$ be the set of linear equations \eqref{summary 2}, and define
\begin{align}
Z_M^K:=[1+\mathcal{B}_{-1}+A_{1}^{\mathcal{L}}+\dots+A_{q-1}^{\mathcal{L}}+B_{1}^{\mathcal{L}}+\dots+B_{p-1}^{\mathcal{L}}]^KZ_M,\notag
\end{align}
the set of equations obtained by taking the equations of $Z_M$ and all their derivatives up to the $K$th order with respect to the differential operators $\mathcal{B}_{-1},A_{1}^{\mathcal{L}},\dots,A_{q-1}^{\mathcal{L}},B_{1}^{\mathcal{L}},\dots,B_{p-1}^{\mathcal{L}}$. The number of equations in $Z_M^K$ is simply $M$ times the number of monomials of degree $K$ chosen from a set of $p+q$ variables, i.e.
\begin{align}
M\binom{p+q+K-1}{K}&=\frac{1}{2}(p+q)\frac{(K+p+q-1)(K+p+q-2)\dots(K+1)}{(p+q-2)!}.\notag
\end{align}
The set of equations $Z_M^K$ is a set of linear equations in the $p+q-2$ unknown functions $X_1,\dots,X_{q-1}$ and $Y_1,\dots,Y_{p-1}$ and at most all their first, second, \dots, $(K+2)^{\text{th}}$ order derivatives with respect to the differential operators $A_{i}^{\mathcal{L}}$, $1\leq i\leq q-1$, $B_{j}^{\mathcal{L}}$, $1\leq j\leq p-1$, and $\mathcal{B}_{-1}$. Let $L$ be the number of unknowns in these equations. Then 
\begin{align}
L&\leq (p+q-2)\binom{p+q+K+2-1}{K+2}=(p+q-2)\frac{(K+p+q+1)(K+p+q)\dots(K+3)}{(p+q-1)!}.\notag
\end{align}
From these considerations, it is clear that a sufficient condition to have $Card(Z_M^K)> L$ is
\begin{align}
\frac{1}{2}(p+q)\frac{(K+p+q-1)!}{K!(p+q-2)!}> (p+q-2)\frac{(K+p+q+1)!}{(K+2)!(p+q-1)!},\notag
\end{align}
or, simplifying this expression, 
\begin{align}
(x^2-3x+4)K^2+(-x^2+3x+4)K-2x(x^2-2x-1)> 0,\notag
\end{align}
where we have noted $x=p+q$. We observe that, with $p$ and $q$ fixed, for $K$ sufficiently large, this inequality is satisfied, since $x^2-3x+4> 0$. Let $K^*$ be the smallest value of $K$ such that this inequality is satisfied, and note $k^*$ the number of equations in $Z_M^{K^*}$, i.e.
\begin{align}
k^*=\frac{1}{2}(p+q)(p+q-1)\binom{p+q+K^*-1}{K^*},\notag
\end{align}
and $L^*$ the number of unknowns in the set of linear equations $Z_M^{K^*}$. Let us note these unknowns $x_1,\dots,x_{L^*}$. Then the system of $k^*$ linear equations that we have obtained can be written
\begin{align}
\big[a_{ij}(f)\big]_{\substack{1\leq i\leq k^*\\ 1\leq j\leq L^*+1}}\left(\begin{array}{llll}1 \\ x_1 \\
\vdots \\
x_{L^*} \end{array}\right)=0.\notag
\end{align}
As $k^*>L^*$, we can select the $L^*+1$ first equations in this system and construct the following system
\begin{align}
\big[a_{ij}(f)\big]_{\substack{1\leq i\leq L^*+1\\ 1\leq j\leq L^*+1}}\left(\begin{array}{llll}1 \\ x_1 \\
\vdots \\
x_{L^*} \end{array}\right)=0.\notag
\end{align}
But then, necessarily, we have
\begin{align}
\det\big[a_{ij}(f)\big]_{\substack{1\leq i\leq L^*+1\\ 1\leq j\leq L^*+1}}=0.\notag
\end{align}
This is a PDE of order $(K^*+3)$ for the function $f$, with variables $a_1,\dots, a_{q-1}$, $b_1,\dots, b_{p-1}$ and $c_1,\dots,c_{2r}$. Since $f=\log\tau_{\vec{m},\vec{n}}^E(0;a,b)=\log P_{p,q}(E;a,b)+\log\tau_{\vec{m},\vec{n}}^\mathbb{R}(0;a,b)$, this yields a nonlinear PDE of order $K^*+3$ in the variables $a_1,\dots,a_{q-1}$, $b_1,\dots,b_{p-1}$ and the endpoints of $E$ for $\log P_{p,q}(E;a,b)$, and thus throuhg formula \eqref{proba-normalized problem} for $\log\mathbb{P}_{b_1,\dots,b_p}^{a_1,\dots,a_q}\big(all\ x_i(t)\in E\big)$, in terms of the input $\log\tau_{\vec{m},\vec{n}}^\mathbb{R}(0;a,b)$, which we think of as known. We thus have proven Theorem \ref{Theorem introduction}.


\section{Non-intersecting Brownian motions with two starting points and two ending points}

In this section we consider a particular situation of the problem studied in the preceding sections. Consider $N$ non-intersecting Brownian motions $x_1(t),x_2(t),\dots,x_N(t)$ in $\mathbb{R}$, conditionned to start at time $t=0$ at two different points and to end up at $t=1$ in two different points, with the coordinates of the starting points and the coordinates of the ending points both satisfying a linear condition. In this particular case, all the results of the preceding sections hold with $p=q=2$. Consequently, by virtue of Theorem \ref{Theorem introduction}, we know that the probability to find all the particles in a certain set $E=\bigcup_{i=1}^{r}[c_{2i-1},c_{2i}]\subset\mathbb{R}$ at a given time $0<t<1$, satisfies a nonlinear PDE of order $6$, the variables being the coordinates of the starting and ending points, and the endpoints of the set $E$. The aim of this section is to improve the result of Theorem \ref{Theorem introduction} in the particular case when $p=q=2$ and to describe this PDE more precisely.

For the sake of clarity, we first recall some notations. Consider $N$ non-intersecting Brownian motions $x_1(t),x_2(t),\dots,x_N(t)$ in $\mathbb{R}$, with $m_1$ particles leaving from $a$ and $m_2$ particles leaving from $-a$, and $n_1$ particles ending in $b$ and $n_2$ particles ending in $-b$. We denote
\begin{align}
\mathbb{P}_{+b,-b}^{+a,-a}\big(\text{all }x_i(t)\in E\big)&:=\mathbb{P}\left(\text{all }x_i(t)\in E\Bigg|\begin{array}{llll}\big(x_1(0),\dots,x_N(0)\big)=\big(\underbrace{a,\dots,a}_{m_1},\underbrace{-a,\dots,-a}_{m_2}\big) \\ \big(x_1(1),\dots,x_N(1)\big)=\big(\underbrace{b,\dots,b}_{n_1},\underbrace{-b,\dots,-b}_{n_2}\big) \end{array}\right),\notag
\end{align}
the probability to find all the particles in a set $E\subset\mathbb{R}$, at a given time $0<t<1$. We have
\begin{align}
\mathbb{P}_{b,-b}^{a,-a}\big(\text{all }x_i(t)\in E\big)=P_{2,2}\Big(\sqrt{\frac{2}{t(1-t)}}\,E;\sqrt{\frac{2(1-t)}{t}}\,a,\sqrt{\frac{2t}{1-t}}\,b\Big),\notag
\end{align}
with the normalized problem defined in (\ref{normalized problem}). We deform $P_{2,2}(E;a,b)$ by adding four families of extra time variables
\begin{align}
t^{(1)}=\big(t_1^{(1)},t_2^{(1)},\dots\big),\quad t^{(2)}=\big(t_1^{(2)},t_2^{(2)},\dots\big),\notag\\
s^{(1)}=\big(s_1^{(1)},s_2^{(1)},\dots\big),\quad s^{(2)}=\big(s_1^{(2)},s_2^{(2)},\dots\big),\notag
\end{align}
and the two parameters $\alpha,\beta\in\mathbb{C}$. We have
\begin{align}
P_{2,2}\big(E;a,b;(t,s),(\alpha,\beta)\big)&=\frac{\tau_{m_1m_2;n_1,n_2}^E(t,s;\alpha,\beta;a,b)}{\tau_{m_1m_2;n_1,n_2}^\mathbb{R}(t,s;\alpha,\beta;a,b)},\label{deformed normalized problem p=q=2}
\end{align}
with $\tau_{m_1m_2;n_1,n_2}^E(t,s;\alpha,\beta;a,b)$ defined in (\ref{tau}), and where $(t,s)=\big(t^{(1)},t^{(2)};s^{(1)},s^{(2)}\big)$. The function $P_{2,2}(E;a,b)$ is then simply given by $P_{2,2}\big(E;a,b;(t,s),(\alpha,\beta)\big)$ evaluated along the locus $\mathcal{L}=\{(t,s)=(0,0),\alpha=\beta=0\}$. We define the function $f:=\log\tau_{m_1,m_2;n_1,n_2}^E(t,s;\alpha,\beta;a,b)$. The following particular version of Theorem \ref{system of equations for log tau with unknowns} holds.
\begin{theorem}\label{particular version p=q=2}
Put $X=-\frac{1}{2}\frac{\partial f}{\partial \alpha}\big|_{\mathcal{L}}$ and $Y=-\frac{1}{2}\frac{\partial f}{\partial \beta}\big|_{\mathcal{L}}$. The function $f=\log\tau_{m_1,m_2;n_1,n_2}^E(t,s;\alpha,\beta;a,b)$ satisfies the following $6$ equations on the locus $\mathcal{L}$
\begin{gather*}
\Big\{A_{1}^{\mathcal{L}}Y,A_{1}^{\mathcal{L}}B_{1}^{\mathcal{L}}f+\frac{m_1}{2}+\frac{n_1}{2}-\frac{N}{4}\Big\}_{A_{1}^{\mathcal{L}}}-\Big\{B_{1}^{\mathcal{L}}X,A_{1}^{\mathcal{L}}B_{1}^{\mathcal{L}}f+\frac{m_1}{2}+\frac{n_1}{2}-\frac{N}{4}\Big\}_{B_{1}^{\mathcal{L}}}=G_{11}^{AB},\\
-\Big\{A_{1}^{\mathcal{L}}Y,A_{1}^{\mathcal{L}}B_{2}^{\mathcal{L}}f+\frac{m_1}{2}+\frac{n_2}{2}-\frac{N}{4}\Big\}_{A_{1}^{\mathcal{L}}}-\Big\{B_{2}^{\mathcal{L}}X,A_{1}^{\mathcal{L}}B_{2}^{\mathcal{L}}f+\frac{m_1}{2}+\frac{n_2}{2}-\frac{N}{4}\Big\}_{B_{2}^{\mathcal{L}}}=G_{12}^{AB},\\
\Big\{A_{2}^{\mathcal{L}}Y,A_{2}^{\mathcal{L}}B_{1}^{\mathcal{L}}f+\frac{m_2}{2}+\frac{n_1}{2}-\frac{N}{4}\Big\}_{A_{2}^{\mathcal{L}}}+\Big\{B_{1}^{\mathcal{L}}X,A_{2}^{\mathcal{L}}B_{1}^{\mathcal{L}}f+\frac{m_2}{2}+\frac{n_1}{2}-\frac{N}{4}\Big\}_{B_{1}^{\mathcal{L}}}=G_{21}^{AB},\\
-\Big\{A_{2}^{\mathcal{L}}Y,A_{2}^{\mathcal{L}}B_{2}^{\mathcal{L}}f+\frac{m_2}{2}+\frac{n_2}{2}-\frac{N}{4}\Big\}_{A_{2}^{\mathcal{L}}}+\Big\{B_{2}^{\mathcal{L}}X,A_{2}^{\mathcal{L}}B_{2}^{\mathcal{L}}f+\frac{m_2}{2}+\frac{n_2}{2}-\frac{N}{4}\Big\}_{B_{2}^{\mathcal{L}}}=G_{22}^{AB},\\
\Big\{A_{2}^{\mathcal{L}}X,A_{1}^{\mathcal{L}}A_{2}^{\mathcal{L}}f+\frac{N}{4}\Big\}_{A_{2}^{\mathcal{L}}}-\Big\{A_{1}^{\mathcal{L}}X,A_{1}^{\mathcal{L}}A_{2}^{\mathcal{L}}f+\frac{N}{4}\Big\}_{A_{1}^{\mathcal{L}}}=G_{12}^{A},\\
\Big\{B_{2}^{\mathcal{L}}Y,B_{1}^{\mathcal{L}}B_{2}^{\mathcal{L}}f+\frac{N}{4}\Big\}_{B_{2}^{\mathcal{L}}}-\Big\{B_{1}^{\mathcal{L}}Y,B_{1}^{\mathcal{L}}B_{2}^{\mathcal{L}}f+\frac{N}{4}\Big\}_{B_{1}^{\mathcal{L}}}=G_{12}^{B},
\end{gather*}
where
\begin{align}
&A_j=A_{j}^{\mathcal{L}}-\Big(\alpha\frac{\partial}{\partial a}+\beta\frac{\partial}{\partial b}\Big),\quad B_j=B_{j}^{\mathcal{L}}-\Big(\alpha\frac{\partial}{\partial a}+\beta\frac{\partial}{\partial b}\Big),\quad 1\leq j\leq 2,\notag
\end{align}
\begin{align}
&A_{j}^{\mathcal{L}}=\frac{1}{2}\Big(\mathcal{B}_{-1}+(-1)^j\frac{\partial}{\partial a}\Big),\quad B_{j}^{\mathcal{L}}=\frac{1}{2}\Big(\mathcal{B}_{-1}+(-1)^j\frac{\partial}{\partial b}\Big),\quad 1\leq j\leq 2,\notag
\end{align}
and where $G_{jk}^{A}$, $G_{jk}^{B}$ and $G_{jk}^{AB}$ only depend on $f$, its derivatives with respect to $a_1,\dots, a_{q-1}$, $b_1,\dots,b_{p-1}$, and its differentials up to the third order with respect to the operators $A_{j}^{\mathcal{L}}$, $B_{j}^{\mathcal{L}}$ and $\hat{\mathcal{B}}_0$, evaluated on the locus $\mathcal{L}$.
\end{theorem}

The equations in Theorem \ref{particular version p=q=2} can be written
\begin{gather}
\big(A_{2}^{\mathcal{L}})^2X-\big(A_{1}^{\mathcal{L}}\big)^2X-\Big(A_{2}^{\mathcal{L}}\log c_{12}^{A}\Big)A_{2}^{\mathcal{L}}X+\Big(A_{1}^{\mathcal{L}}\log c_{12}^{A}\Big)A_{1}^{\mathcal{L}}X=g_{12}^{A},\notag\\  \big(B_{2}^{\mathcal{L}})^2Y-\big(B_{1}^{\mathcal{L}}\big)^2Y-\Big(B_{2}^{\mathcal{L}}\log c_{12}^{B}\Big)B_{2}^{\mathcal{L}}Y+\Big(B_{1}^{\mathcal{L}}\log c_{12}^{B}\Big)B_{1}^{\mathcal{L}}Y=g_{12}^{B},\notag\\
\big(A_{1}^{\mathcal{L}}\big)^2Y-\big(B_{1}^{\mathcal{L}}\big)^2X-\Big(A_{1}^{\mathcal{L}}\log c_{11}^{AB}\Big)A_{1}^{\mathcal{L}}Y+\Big(B_{1}^{\mathcal{L}}\log c_{11}^{AB}\Big)B_{1}^{\mathcal{L}}X=g_{11}^{AB},\label{the six equations}\\
\big(A_{2}^{\mathcal{L}}\big)^2Y+\big(B_{1}^{\mathcal{L}}\big)^2X-\Big(A_{2}^{\mathcal{L}}\log c_{21}^{AB}\Big)A_{2}^{\mathcal{L}}Y-\Big(B_{1}^{\mathcal{L}}\log c_{21}^{AB}\Big)B_{1}^{\mathcal{L}}X=g_{21}^{AB},\notag\\
-\big(A_{1}^{\mathcal{L}}\big)^2Y-\big(B_{2}^{\mathcal{L}}\big)^2X+\Big(A_{1}^{\mathcal{L}}\log c_{12}^{AB}\Big)A_{1}^{\mathcal{L}}Y+\Big(B_{2}^{\mathcal{L}}\log c_{12}^{AB}\Big)B_{2}^{\mathcal{L}}X=g_{12}^{AB},\notag\\
-\big(A_{2}^{\mathcal{L}}\big)^2Y+\big(B_{2}^{\mathcal{L}}\big)^2X+\Big(A_{2}^{\mathcal{L}}\log c_{22}^{AB}\Big)A_{2}^{\mathcal{L}}Y-\Big(B_{2}^{\mathcal{L}}\log c_{22}^{AB}\Big)B_{2}^{\mathcal{L}}X=g_{22}^{AB},\notag
\end{gather}
where the $c_{ij}$'s and the $g_{ij}$'s are defined in (\ref{c_ij}) and (\ref{g_ij}) with $p=q=2$.

The four differential operators $A_{1}^{\mathcal{L}},A_{2}^{\mathcal{L}},B_{1}^{\mathcal{L}},B_{2}^{\mathcal{L}}$ are not linearly independent. Indeed, we have
\begin{align}
&A_{1}^{\mathcal{L}}+A_{2}^{\mathcal{L}}=B_{1}^{\mathcal{L}}+B_{2}^{\mathcal{L}}=\mathcal{B}_{-1},\qquad A_{2}^{\mathcal{L}}-A_{1}^{\mathcal{L}}=\frac{\partial}{\partial a},\qquad B_{2}^{\mathcal{L}}-B_{1}^{\mathcal{L}}=\frac{\partial}{\partial b}.\notag
\end{align}
Consequently, we will rewrite the six equations (\ref{the six equations}) in terms of the three independent, commuting differential operators $\frac{\partial}{\partial a}$, $\frac{\partial}{\partial b}$ and $\mathcal{B}_{-1}$. Before performing this, let us introduce some notations. First, we will write
\begin{align}
\frac{\partial F}{\partial a}=F_a,\quad \frac{\partial F}{\partial b}=F_b, \quad \mathcal{B}_{-1}=F_c,\notag
\end{align}
for a function $F$. Next, we put
\begin{gather*}
G_1:=g_{12}^{A},\quad G_2:=g_{12}^{B},\quad G_3:=-g_{11}^{AB}+g_{21}^{AB}-g_{12}^{AB}+g_{22}^{AB},\quad G_4:=g_{11}^{AB}+g_{21}^{AB}-g_{12}^{AB}-g_{22}^{AB},\\
G_5:=\frac{1}{2}\big(g_{11}^{AB}-g_{21}^{AB}-g_{12}^{AB}+g_{22}^{AB}\big),\quad G_6:=-\frac{1}{2}\big(g_{11}^{AB}+g_{21}^{AB}+g_{12}^{AB}+g_{22}^{AB}\big),
\end{gather*}
and 
\begin{gather*}
\Delta_1:=\log c_{12}^A,\quad \Delta_2:=\log c_{12}^B,\quad \Delta_{3}:=\log \frac{c_{11}^{AB}c_{21}^{AB}}{c_{12}^{AB}c_{22}^{AB}},\\
\Delta_{4}:=\log\frac{c_{11}^{AB}c_{12}^{AB}}{c_{21}^{AB}c_{22}^{AB}},\quad \Delta_{5}:=\log\frac{c_{21}^{AB}c_{12}^{AB}}{c_{11}^{AB}c_{22}^{AB}},\quad  \Delta_{6}:=\log\big(c_{11}^{AB}c_{21}^{AB}c_{12}^{AB}c_{22}^{AB}\big).
\end{gather*}
We then define
\begin{gather*}
\alpha:=\frac{1}{4}(-\Delta_{3c}+\Delta_{6b}),\qquad\beta:=\frac{1}{4}(\Delta_{6c}-\Delta_{3b}),\\ 
\gamma:=\frac{1}{4}(\Delta_{4c}+\Delta_{5b}),\qquad \delta:=\frac{1}{4}(\Delta_{5c}+\Delta_{4b}),
\end{gather*}
and
\begin{gather*}
\hat\alpha:=\frac{1}{4}(-\Delta_{4c}+\Delta_{6a}),\qquad\hat\beta:=\frac{1}{4}(\Delta_{6c}-\Delta_{4a}),\\
\hat\gamma:=\frac{1}{4}(\Delta_{3c}+\Delta_{5a}),\qquad \hat\delta:=\frac{1}{4}(\Delta_{5c}+\Delta_{3a}).
\end{gather*}

Taking adequate linear combinations of the equations (\ref{the six equations}), and using all these notations, we obtain the following equivalent system
\begin{gather}
X_{ac}=G_1+\frac{1}{2}\Delta_{1a}X_c+\frac{1}{2}\Delta_{1c}X_a,\notag\\
Y_{bc}=G_2+\frac{1}{2}\Delta_{2b}Y_c+\frac{1}{2}\Delta_{2c}Y_b,\notag\\
X_{cc}+X_{bb}=G_{3}+\beta X_c+\alpha X_b+\hat\delta Y_c+\hat\gamma Y_a,\notag\\
Y_{cc}+Y_{aa}=G_{4}+\delta X_c+\gamma X_b+\hat\beta Y_c+\hat\alpha Y_a,\label{six equations in normal form bis}\\
X_{bc}-Y_{ac}=G_{5}+\frac{\alpha}{2} X_c+\frac{\beta}{2} X_b-\frac{\hat\alpha}{2} Y_c-\frac{\hat\beta}{2} Y_a,\notag\\
0=G_{6}+\frac{\gamma}{2} X_c+\frac{\delta}{2} X_b-\frac{\hat\gamma}{2} Y_c-\frac{\hat\delta}{2} Y_a.\notag
\end{gather}
This is a system of linear equations in the variables
\begin{align}
X_a,X_b,X_c,X_{ac},X_{bc},X_{bb},X_{cc},Y_a,Y_b,Y_c,Y_{ac},Y_{bc},Y_{aa},Y_{cc}.\notag
\end{align}
The coefficients of this linear system depend on $\log\tau^{E}_{m_1,m_2;n_1,n_2}(0;a,b)$ and its partial derivatives up to the third order with respect to $a,b$ and the endpoints of $E$. They are highly related. Indeed, there are only twelve non zero independent coefficients $\alpha,\beta,\gamma,\delta,\hat\alpha,\hat\beta,\hat\gamma,\hat\delta,\Delta_{1b},\Delta_{1c},\Delta_{2a},\Delta_{2c}$. We will now generate new linear equations by applying successively the derivatives $\frac{\partial}{\partial a}$, $\frac{\partial}{\partial b}$ and $\mathcal{B}_{-1}=\frac{\partial}{\partial c}$ to this system. Consider the following system of $37$ equations
\begin{align}
&(\ref{six equations in normal form bis}.1),(\ref{six equations in normal form bis}.2),(\ref{six equations in normal form bis}.3),(\ref{six equations in normal form bis}.4),(\ref{six equations in normal form bis}.5),(\ref{six equations in normal form bis}.6),(\ref{six equations in normal form bis}.6)_a,(\ref{six equations in normal form bis}.6)_b,(\ref{six equations in normal form bis}.6)_c,\notag\\
&(\ref{six equations in normal form bis}.1)_a,(\ref{six equations in normal form bis}.1)_b,(\ref{six equations in normal form bis}.1)_c,(\ref{six equations in normal form bis}.2)_a,(\ref{six equations in normal form bis}.2)_b,(\ref{six equations in normal form bis}.2)_c,(\ref{six equations in normal form bis}.3)_a,(\ref{six equations in normal form bis}.3)_b,(\ref{six equations in normal form bis}.3)_c,(\ref{six equations in normal form bis}.4)_a,(\ref{six equations in normal form bis}.4)_b,(\ref{six equations in normal form bis}.4)_c,\notag\\
&(\ref{six equations in normal form bis}.5)_a,(\ref{six equations in normal form bis}.5)_b,(\ref{six equations in normal form bis}.5)_c,(\ref{six equations in normal form bis}.6)_{aa},(\ref{six equations in normal form bis}.6)_{ab},(\ref{six equations in normal form bis}.6)_{ac},(\ref{six equations in normal form bis}.6)_{bb},(\ref{six equations in normal form bis}.6)_{bc},(\ref{six equations in normal form bis}.6)_{cc},\notag\\
&(\ref{six equations in normal form bis}.1)_{bb}-(\ref{six equations in normal form bis}.2)_{aa}-(\ref{six equations in normal form bis}.5)_{ab},\notag\\
&(\ref{six equations in normal form bis}.1)_{bb}+(\ref{six equations in normal form bis}.1)_{cc}-(\ref{six equations in normal form bis}.3)_{ac},\notag\\
&(\ref{six equations in normal form bis}.2)_{aa}+(\ref{six equations in normal form bis}.2)_{cc}-(\ref{six equations in normal form bis}.4)_{bc},\notag\\
&(\ref{six equations in normal form bis}.5)_{aa}+(\ref{six equations in normal form bis}.5)_{bb}+(\ref{six equations in normal form bis}.5)_{cc}+(\ref{six equations in normal form bis}.2)_{ab}-(\ref{six equations in normal form bis}.1)_{ab}+(\ref{six equations in normal form bis}.4)_{ac}-(\ref{six equations in normal form bis}.3)_{bc},\notag\\
&(\ref{six equations in normal form bis}.6)_{abc}-\frac{\delta}{2}\times(\ref{six equations in normal form bis}.1)_{bb}-\frac{\gamma}{2}\times(\ref{six equations in normal form bis}.1)_{bc}+\frac{\hat\delta}{2}\times(\ref{six equations in normal form bis}.2)_{aa}+\frac{\hat\gamma}{2}\times(\ref{six equations in normal form bis}.2)_{ac},\notag\\
&(\ref{six equations in normal form bis}.6)_{aac}-\frac{\delta}{2}\times(\ref{six equations in normal form bis}.1)_{ab}-\frac{\gamma}{2}\times(\ref{six equations in normal form bis}.1)_{ac}+\frac{\hat\delta}{2}\times\big((\ref{six equations in normal form bis}.1)_{ab}-(\ref{six equations in normal form bis}.5)_{aa}\big)+\frac{\hat\gamma}{2}\times\big((\ref{six equations in normal form bis}.1)_{bc}-(\ref{six equations in normal form bis}.5)_{ac}\big),\notag\\
&(\ref{six equations in normal form bis}.6)_{bbc}-\frac{\delta}{2}\times\big((\ref{six equations in normal form bis}.5)_{bb}+(\ref{six equations in normal form bis}.2)_{ab}\big)-\frac{\gamma}{2}\times\big((\ref{six equations in normal form bis}.5)_{bc}+(\ref{six equations in normal form bis}.2)_{ac}\big)+\frac{\hat\delta}{2}\times(\ref{six equations in normal form bis}.2)_{ab}+\frac{\hat\gamma}{2}\times(\ref{six equations in normal form bis}.2)_{bc}.\notag
\end{align}
These are linear equations, the variables being all the first, second and third order derivatives in $a,b,c$ of $X$ and $Y$, except $X_{aaa}$ and $Y_{bbb}$. Consequently, there are $36$ variables. Constructing the vector $\vec{x}:=(1,X_a,X_b,X_c,Y_a,Y_b,Y_c,\dots)^T\in\mathbb{C}^{37}$ (the first component being one, followed by the $36$ variables), this system of linear equations can be written
\begin{align}
\big[a_{ij}(f)\big]_{1\leq i,j\leq 37}.\vec{x}=0,\notag
\end{align}
where $a_{ij}$ are differential operators of order less or equal then $6$. But then we have necessarily that
\begin{align}
\det\big[a_{ij}(f)\big]_{1\leq i,j\leq 37}=0.\notag
\end{align}
This is a PDE for $f=\log P_{2,2}(E;a,b)+\log \tau_{m_1,m_2;n_1,n_2}^{\mathbb{R}}(0;0;a,b)$. Thus using the structure of the $6$ equations in Theorem \ref{particular version p=q=2}, we have obtained a much better result than the one in Theorem \ref{Theorem introduction}. Indeed, performing in detail the general method described in the proof of Theorem \ref{Theorem introduction}, one obtains in the case when $p=q=2$ a PDE given by a determinant of a $107\times 107$ matrix that is equal to zero. Thus in any particular case, one can do much better than the general case in Theorem \ref{Theorem introduction}.

\appendix

\section{The integral over the full range}

In section 5 we have shown that the function $f=\log\tau_{\vec{m},\vec{n}}^E(0;a,b)=\log P_{p,q}(E;a,b)+\log\tau_{\vec{m},\vec{n}}^\mathbb{R}(0;a,b)$ satisfies a nonlinear PDE, and in section 6 we have described this PDE when $p=q=2$. To obtain a PDE for $\log P_{p,q}(E;a,b)$ it is necessary to evaluate $\log\tau_{\vec{m},\vec{n}}^\mathbb{R}(0;a,b)$. This has been done in the particular case when $p=1$ (see the appendix in \cite{A.V.M.1}), but it seems harder to evaluate this function when $p,q\geq 2$. In this appendix, we conjecture some results about the evaluation of $\log\tau_{\vec{m},\vec{n}}^\mathbb{R}(0;a,b)$ when $p=q=2$.\\

Consider $N$ non-intersecting Brownian motions $x_1(t),x_2(t),\dots,x_N(t)$ in $\mathbb{R}$, with $m_1$ particles leaving from $a$ and $m_2$ particles leaving from $-a$, and $n_1$ particles ending in $b$ and $n_2$ particles ending in $-b$. We suppose $m_1=n_1$ and $m_2=n_2$. We denote
\begin{align}
\mathbb{P}_{+b,-b}^{+a,-a}\big(\text{all }x_i(t)\in E\big)&:=\mathbb{P}\left(\text{all }x_i(t)\in E\Bigg|\begin{array}{llll}\big(x_1(0),\dots,x_N(0)\big)=\big(\underbrace{a,\dots,a}_{m_1},\underbrace{-a,\dots,-a}_{m_2}\big) \\ \big(x_1(1),\dots,x_N(1)\big)=\big(\underbrace{b,\dots,b}_{m_1},\underbrace{-b,\dots,-b}_{m_2}\big) \end{array}\right),\notag
\end{align}
the probability to find all the particles in a set $E\subset\mathbb{R}$, at a given time $0<t<1$. We have
\begin{align}
\mathbb{P}_{b,-b}^{a,-a}\big(\text{all }x_i(t)\in E\big)=P_{2,2}\big(\tilde E;\tilde a,\tilde b\big),\notag
\end{align}
with the normalized problem defined in (\ref{normalized problem}), and where
\begin{align}
&\tilde{a}:=\sqrt{\frac{2(1-t)}{t}}\,a,\qquad\tilde{b}:=\sqrt{\frac{2t}{1-t}}\,b,\qquad\tilde{E}:=\sqrt{\frac{2}{t(1-t)}}\,E.\notag
\end{align}
Notice that $\tilde{a}\tilde{b}=2ab$. As shown in (\ref{deformed normalized problem}), we have
\begin{align}
P_{2,2}\big(\tilde E;\tilde a,\tilde b\big)=\frac{\tau_{\vec{m},\vec{n}}^E(0;\tilde a,\tilde b)}{\tau_{\vec{m},\vec{n}}^\mathbb{R}(0;\tilde a,\tilde b)},\notag
\end{align}
with $\tau_{\vec{m},\vec{n}}^E(0;\tilde a,\tilde b)$ defined in (\ref{tau}). We try to evaluate the function
\begin{align}
\tau_{\vec{m},\vec{n}}^\mathbb{R}(0;\tilde{a},\tilde b)&=\det\left(\begin{array}{cc}
\Big(\int_{\mathbb{R}}x^{i+j}e^{(\tilde a+\tilde b)x}e^{-\frac{x^2}{2}}dx\Big)_{\substack{0\leq i< m_1 \\ 0\leq j< m_1}} &  \Big(\int_{\mathbb{R}}x^{i+j}e^{(\tilde a-\tilde b)x}e^{-\frac{x^2}{2}}dx\Big)_{\substack{0\leq i< m_1 \\ 0\leq j< m_2}} \\
\Big(\int_{\mathbb{R}} x^{i+j}e^{(-\tilde a+\tilde b)x}e^{-\frac{x^2}{2}}dx\Big)_{\substack{0\leq i< m_2 \\ 0\leq j< m_1}} & \Big(\int_{\mathbb{R}}x^{i+j}e^{(-\tilde a-\tilde b)x}e^{-\frac{x^2}{2}}dx\Big)_{\substack{0\leq i< m_2 \\ 0\leq j< m_2}} 
\end{array}\right).\notag
\end{align}

Define 
\begin{align}
\mu_{i+j}(c)=\int_{\mathbb{R}}x^{i+j}e^{-\frac{x^2}{2}+cx}dx.\notag
\end{align}
We have
\begin{align}
\mu_0(c)=\sqrt{2\pi}e^{\frac{c^2}{2}},\quad\text{and}\quad\mu_{i+j}(c)=\Big(\frac{d}{d c}\Big)^{i+j}\mu_0(c).\notag
\end{align}
Define also the polynomials (Hermite polynomials up to a slight change of variables)
\begin{align}
p_j(x)=e^{-\frac{x^2}{2}}\Big(\frac{d}{d x}\Big)^{j}e^{\frac{x^2}{2}}.\notag
\end{align}
We then have
\begin{align}
\tau_{\vec{m},\vec{n}}^\mathbb{R}(0;\tilde a,\tilde b)&=\det\left(\begin{array}{cc}
\big(\mu_{i+j}(\tilde a+\tilde b)\big)_{\substack{0\leq i\leq m_1-1 \\ 0\leq j\leq m_1-1}} &  \big(\mu_{i+j}(\tilde a-\tilde b)\big)_{\substack{0\leq i\leq m_1-1 \\ 0\leq j\leq m_2-1}} \\
\big(\mu_{i+j}(-\tilde a+\tilde b)\big)_{\substack{0\leq i\leq m_2-1 \\ 0\leq j\leq m_1-1}} & \big(\mu_{i+j}(-\tilde a-\tilde b)\big)_{\substack{0\leq i\leq m_2-1 \\ 0\leq j\leq m_2-1}} 
\end{array}\right)\notag\\
&=(2\pi)^{\frac{m_1+m_2}{2}}e^{\frac{m_1+m_2}{2}(\tilde a+\tilde b)^2}e^{-4m_1\tilde a\tilde b}\notag\\
&\phantom{=(2\pi)^{\frac{m_1+m_2}{2}}}\times\,\det\left(\begin{array}{cc}
\big(e^{4\tilde a\tilde b}p_{i+j}(\tilde a+\tilde b)\big)_{\substack{0\leq i\leq m_1-1 \\ 0\leq j\leq m_1-1}} &  \big(p_{i+j}(\tilde a-\tilde b)\big)_{\substack{0\leq i\leq m_1-1 \\ 0\leq j\leq m_2-1}} \\
\big(p_{i+j}(-\tilde a+\tilde b)\big)_{\substack{0\leq i\leq m_2-1 \\ 0\leq j\leq m_1-1}} & \big(p_{i+j}(-\tilde a-\tilde b)\big)_{\substack{0\leq i\leq m_2-1 \\ 0\leq j\leq m_2-1}} 
\end{array}\right).\notag
\end{align}
We will use the following lemma.
\begin{lemma}
Consider the block matrix
\begin{align}
\left(\begin{array}{cc}
A & B \\
C & D 
\end{array}\right),\notag
\end{align}
with $A,D$ square matrices, and $D$ invertible. Then
\begin{align}
\det\left(\begin{array}{cc}
A & B \\
C & D 
\end{array}\right)=\det D\times \det\big(A-BD^{-1}C\big).\notag
\end{align}
\end{lemma}

\begin{proof}
Doing row and column perations, we have
\begin{align}
\det\left(\begin{array}{cc}
A & B \\
C & D 
\end{array}\right)&=\det\left(\begin{array}{cc}
I & 0 \\
0 & D 
\end{array}\right)\times \det\left(\begin{array}{cc}
A & B \\
D^{-1}C & I 
\end{array}\right)=\det D\times \det\left(\begin{array}{cc}
A-BD^{-1}C & 0 \\
D^{-1}C & I
\end{array}\right)\notag\\
&=\det D\times \det\big(A-BD^{-1}C\big).\notag
\end{align}
\end{proof}

Using this lemma, we have
\begin{align}
\tau_{\vec{m},\vec{n}}^\mathbb{R}(0;\tilde a,\tilde b)=(2\pi)^{\frac{m_1+m_2}{2}}e^{\frac{m_1+m_2}{2}(\tilde a+\tilde b)^2}e^{-4m_1\tilde a\tilde b}\det D\times \det\big(A-BD^{-1}C\big),\notag
\end{align}
where
\begin{gather}
A=\Big(e^{4\tilde a\tilde b}p_{i+j}(\tilde a+\tilde b)\Big)_{\substack{0\leq i\leq m_1-1 \\ 0\leq j\leq m_1-1}},\qquad B=\Big(p_{i+j}(\tilde a-\tilde b)\Big)_{\substack{0\leq i\leq m_1-1 \\ 0\leq j\leq m_2-1}},\notag\\
C=\Big(p_{i+j}(-\tilde a+\tilde b)\Big)_{\substack{0\leq i\leq m_2-1 \\ 0\leq j\leq m_1-1}},\qquad D=\Big(p_{i+j}(-\tilde a-\tilde b)\Big)_{\substack{0\leq i\leq m_2-1 \\ 0\leq j\leq m_2-1}},\label{ABCD}
\end{gather}
and it is well-known that $\det D=\prod_{i=0}^{m_2-1}i\,!$. Let us note
\begin{align}
X=e^{4\tilde a\tilde b}-\sum_{j=0}^{m_2-1}\frac{(4\tilde a\tilde b)^j}{j\,!}=(4\tilde a\tilde b)^{m_2}\frac{1}{2\pi i}\oint_{\Gamma_{0,4\tilde a\tilde b}}\frac{e^z\,dz}{z^{m_2}(z-4\tilde a\tilde b)}=(8ab)^{m_2}\frac{1}{2\pi i}\oint_{\Gamma_{0,8ab}}\frac{e^z\,dz}{z^{m_2}(z-8ab)},\label{X}
\end{align}
where $\Gamma_{0,4\tilde a\tilde b}$ denotes a closed contour containing $0$ and $4\tilde a\tilde b$ in the complex plane. Computer observations point out that for $A,B,C,D$ as in (\ref{ABCD}) we have
\begin{align}
\det\big(A-BD^{-1}C\big)=\det\Big(p_{i+j}(\tilde a+\tilde b)X+P_{i,j}(\tilde a,\tilde b)\Big)_{\substack{0\leq i\leq m_1-1 \\ 0\leq j\leq m_1-1}},\label{det} 
\end{align}
where $P_{i,j}(\tilde a,\tilde b)$ is a polynomial in $\tilde a,\tilde b$ such that $P_{i,j}(\tilde a,\tilde b)=P_{j,i}(\tilde b,\tilde a)$, $P_{0,0}(\tilde a,\tilde b)=0$, and the order of $P_{i,j}(\tilde a,\tilde b)$ is $2(m_1-1)+i+j$ when $i+j>0$. We will develop (\ref{det}) in the large $m_2$-limit.\\

\begin{figure}[!ht]
\centering
\includegraphics[width=7cm]{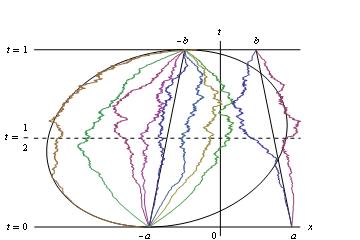}
\caption{Non-intersecting Brownian motions in the large $m_2$-limit, with $m_1$ fixed}\label{picture}
\end{figure}

Let us consider the large $m_2$-limit, keeping $m_1$ fixed, see Figure \ref{picture}. If $m_1=0$, for large $m_2$ the mean density of brownian particles at any time $0<t<1$ is supported by an interval with endpoints given by $\pm\sqrt{2m_2t(1-t)}-a(1-t)-bt$. When $m_1$ is fixed but not necessarily zero, the non-intersecting nature of the cloud of $m_2$ particles implies that the largest one will again reach a height of about $\sqrt{\frac{m_2}{2}}-\frac{a+b}{2}$ at $t=\frac{1}{2}$. We will consider the following scaling given in \cite{A.F.V.M.} for the starting and the target points
\begin{align}
a=\frac{1}{2}\sqrt{\frac{m_2}{2}}\Big(1+\frac{A}{m_2^{1/3}}\Big),\quad\text{and}\quad b=\frac{1}{2}\sqrt{\frac{m_2}{2}}\Big(1-\frac{B}{m_2^{1/3}}\Big).
\end{align}
With this scaling, the $m_1$ wanderers will interact with the bulk of $m_2$ particles ($m_2$ very large), upon considering regions close to $x=\sqrt{\frac{m_2}{2}}-\frac{a+b}{2}$ and $t=\frac{1}{2}$, namely at space-time positions $(x,t)$ which scale like
\begin{align}
&t=\frac{1}{2}+\frac{1}{2}\frac{\tau}{m_2^{1/3}},\qquad x=\frac{1}{2}\sqrt{\frac{m_2}{2}}+\frac{\xi-\tau^2}{2\sqrt{2}m_2^{1/6}}+\frac{1}{4\sqrt{2}}m_2^{1/6}(B-A)+\frac{1}{4\sqrt{2}}\frac{(A+B)\tau}{m_2^{1/6}}.\notag
\end{align}
We suppose $A<B$. Under this scaling, the quantity
\begin{align}
8ab=m_2\Big(1+\frac{A-B}{m_2^{1/3}}-\frac{AB}{m_2^{2/3}}\Big),\notag
\end{align}
is strictly less than $m_2$, for $m_2$ large enough. Consequently, by Cauchy's theorem, the contour $\Gamma_{0,8ab}$ in (\ref{X}) can be taken to be a circle centered at the origin and of radius $m_2$. We will follow \cite{A.F.V.M.} to obtain an assymptotic expansion for $X$. Making the change of variable $z=um_2$ in the integral definig $X$ we have
\begin{align}
X&=(8ab)^{m_2}\frac{m_2^{-m_2}}{2\pi i}\oint_{|u|=1}\frac{e^{m_2F(u)}}{u-1-(A-B)m_2^{-1/3}+ABm_2^{-2/3}}du,\label{integral}
\end{align}
where
\begin{align}
F(u):=u-\ln u=1+\frac{1}{2}(u-1)^2+O\big((u-1)^3\big),\notag
\end{align}
with 
\begin{align}
\Re\big(F(u)\big)=\Re(u)-\ln(|u|).\label{Real part of F}
\end{align}
The stationary points of the function $F(u)$ are solution of the equation $F'(u)=0$, and thus there is one stationary point at $u=1$. We can deform the path $|u|=1$ into $\gamma_\delta=\{1+iy|-\delta\leq y\leq \delta\}$ plus a circle segment $\gamma'$ centered at the origin and joining the extremities of $\gamma_\delta$. It follows from (\ref{Real part of F}) that $\gamma_\delta$ is tangent to the steep descent path through $u=1$. We choose $\delta=m_2^{-2/5}$. Then the contribution to the integral coming from $\gamma_\delta$ is given by
\begin{align}
\int_{\gamma_\delta}\frac{e^{m_2F(u)}}{u-1-(A-B)m_2^{-1/3}+ABm_2^{-2/3}}du=\frac{-m_2^{1/3}e^{m_2}}{(A-B)}\int_{1-im_2^{-2/5}}^{1+im_2^{-2/5}}e^{\frac{m_2}{2}(u-1)^2}du\big(1+\mathcal{O}(m_2^{-1/5})\big).\notag
\end{align}
Making the change of variable $\omega=-i\sqrt{\frac{m_2}{2}}(u-1)$, we obtain
\begin{align}
\int_{\gamma_\delta}\frac{e^{m_2F(u)}}{u-1-(A-B)m_2^{-1/3}+ABm_2^{-2/3}}du=\frac{-i\sqrt{2}m_2^{-1/6}e^{m_2}}{(A-B)}\int_{-\frac{1}{\sqrt{2}}m_2^{1/10}}^{\frac{1}{\sqrt{2}}m_2^{1/10}}e^{-\omega^2}d\omega\big(1+\mathcal{O}(m_2^{-1/5})\big).\notag
\end{align}
As
\begin{align}
\int_{\frac{1}{\sqrt{2}}m_2^{1/10}}^{+\infty}e^{-\omega^2}d\omega=o\big(m_2^{-1/5}\big),\notag
\end{align}
we have
\begin{align}
\int_{\gamma_\delta}\frac{e^{m_2F(u)}}{u-1-(A-B)m_2^{-1/3}+ABm_2^{-2/3}}du=\frac{-i\sqrt{2\pi}m_2^{-1/6}e^{m_2}}{(A-B)}\big(1+\mathcal{O}(m_2^{-1/5})\big).\notag
\end{align}

Let us now evaluate the contribution to the integral coming from $\gamma'$. Along $\gamma'$, we have $u=\sqrt{1+\delta^2}e^{i\theta}$ with $|\cos\theta|\leq \frac{1}{\sqrt{1+\delta^2}}$ and $\delta=m_2^{-2/5}$, and thus
\begin{align}
\big|e^{m_2F(u)}\big|=e^{m_2\Re(F(u))}=e^{m_2\big(\sqrt{1+\delta^2}\cos\theta-\frac{1}{2}\ln(1+\delta^2)\big)}\leq e^{m_2}e^{-\frac{m_2}{2}\ln(1+\delta^2)}.\notag
\end{align}
It follows that 
\begin{align}
e^{-m_2}\big|e^{m_2F(u)}\big|= \mathcal{O}\Big(e^{-\frac{1}{2}m_2^{1/5}}\Big).\notag
\end{align}
Along $\gamma'$, we also have
\begin{align}
\Big|\frac{1}{u-1-(A-B)m_2^{-1/3}+ABm_2^{-2/3}}\Big|&\leq\frac{1}{\big|\sqrt{1+\delta^2}-1-|A-B|m_2^{-1/3}-|AB|m_2^{-2/3}\big|},\notag
\end{align}
and thus
\begin{align}
\Big|\frac{1}{u-1-(A-B)m_2^{-1/3}+ABm_2^{-2/3}}\Big|\leq\frac{-m_2^{1/3}}{A-B}\big(1+\mathcal{O}(m_2^{-1/3})\big).\notag
\end{align}
It follows that the contribution to the integral in (\ref{integral}) from $\gamma'$ is of order $O(e^{-cm_2^{1/5}})$ smaller then the main contribution coming from $\gamma_\delta$, for some $0<c<\frac{1}{2}$. Consequently we have 
\begin{align}
\oint_{\Gamma_{0,4\tilde a\tilde b}}\frac{e^z\,dz}{z^{m_2}(z-4\tilde a\tilde b)}=-\sqrt{2\pi} im_2^{-1/6}\Big(\frac{e}{m_2}\Big)^{m_2}\frac{1}{(A-B)}\big(1+\mathcal{O}(m_2^{-1/5})\big).\notag
\end{align}
It follows that in the large $m_2$-limit
\begin{align}
X=\frac{-m_2^{-1/6}}{\sqrt{2\pi}(A-B)}\exp\Big(m_2+(A-B)m_2^{2/3}-\frac{1}{2}(A^2+B^2)m_2^{1/3}+\frac{1}{2}AB(A-B)\Big)\big(1+\mathcal{O}(m_2^{-1/5})\big).\notag
\end{align}
Consequently, in the large $m_2$-limit, for $m_1$ fixed, $\tau_{\vec{m},\vec{n}}^\mathbb{R}(\tilde a,\tilde b)$ is expected to behave like
\begin{align}
\tau_{\vec{m},\vec{n}}^\mathbb{R}(\tilde a,\tilde b)\approx(2\pi)^{\frac{m_1+m_2}{2}}e^{\frac{m_1+m_2}{2}(\tilde a+\tilde b)^2}e^{-4m_2\tilde a\tilde b}\Big(\prod_{i=0}^{m_1-1}i\,!\Big)\Big(\prod_{j=0}^{m_2-1}j\,!\Big)X^{m_2}.\notag
\end{align}


\end{document}